\documentclass[final,leqno]{siamltex}

\usepackage{amsmath,amssymb,amscd,amsxtra,amsfonts,mathrsfs}
\usepackage{epsf,graphicx,epsfig,color,latexsym,cite,cases}

\title{Analysis of Transient Acoustic-Elastic Interaction in an Unbounded
Structure}

\author{Yixian Gao\thanks{School of Mathematics and Statistics, Center for
Mathematics and Interdisciplinary Sciences, Northeast Normal University,
Changchun, Jilin 130024, China. This author's research was supported in part by
NSFC grant 11571065 and Jilin Science and Technology Development Project. ({\tt
gaoyx643@nenu.edu.cn})} \and Peijun Li\thanks{Department of Mathematics, Purdue
University, West Lafayette, Indiana 47907, USA. This author's research was
supported in part by the NSF grant DMS-1151308.
({\tt lipeijun@math.purdue.edu})} \and Bo Zhang\thanks{LSEC and Institute of
Applied Mathematics, AMSS, Chinese Academy of Sciences, Beijing 100190, China.
This author's research was partially supported by the
NNSFC grants 61379093 and 91430102. ({\tt b.zhang@amt.ac.cn})}}

\begin{document}

\maketitle

\begin{abstract}
Consider the wave propagation in a two-layered medium consisting of a homogeneous
compressible air or fluid on top of a homogeneous isotropic elastic solid. The
interface between the two layers is assumed to be an unbounded rough surface.
This paper concerns the time-domain analysis of such an acoustic-elastic
interaction problem in an unbounded structure in three dimensions.
Using an exact transparent boundary condition and suitable interface conditions,
we study an initial-boundary value problem for the coupling of the Helmholtz
equation and the Navier equation. The well-posedness and stability are
established for the reduced problem. Our proof is based on the method of energy,
the Lax--Milgram lemma, and the inversion theorem of the Laplace transform.
Moreover, a priori estimates with explicit dependence on the time are achieved
for the quantities of acoustic pressure and elastic displacement by taking
special test functions for the time-domain variational problem.
\end{abstract}



\pagestyle{myheadings}
\thispagestyle{plain}
\markboth{Y. Gao, P. Li, and B. Zhang}{Transient Acoustic-Elastic Interaction}

\section{Introduction}

Consider a two-layered medium which consists of a homogeneous compressible air
or fluid on top of a homogeneous isotropic elastic solid. The interface between
air/fluid and solid is assumed to be an unbounded rough surface. An unbounded
rough surface refers to a non-local perturbation of an infinite plane surface
such that the whole surface lies within a finite distance of the original plane.
As a source located in the solid, the external force generates an elastic
wave, which propagates towards the interface and further excites an acoustic
wave in the air/fluid. This process leads to an air/fluid-solid interaction
problem with an unbounded interface separating the acoustic and elastic waves
which are coupled on the interface through two continuity conditions. The
first kinematic interface condition is imposed to ensure that the normal
velocity of the air/fluid on one side of the boundary matches the accelerated
velocity of the solid on another side. The second one is the dynamic condition
which results from the balance of forces on two sides of the interface. The
model problem describes the seismic wave propagation in the air/fluid-solid
medium due to the excitation of an earthquake source which is located in the
crust between the lithosphere and the mantle of the Earth. The goal of this
paper is to carry the mathematical analysis of the time-domain acoustic-elastic
scattering problem in such an unbounded structure in three dimensions.

This problem falls into the class of unbounded rough surface scattering
problems, which have been of great interest to physicists, engineers, and
applied mathematicians for many years due to their wide range of applications in
optics, acoustics, radio-wave propagation, seismology, and radar techniques
\cite{Abubakar1962, Elfouhaily2004, Ogilvy191, Voronovich1998,
Warnick2001}. The elastic wave scattering by unbounded interfaces has many
important applications in geophysics and seismology. For instance, the problem
of elastic pulse transmission and reflection through the Earth is fundamental to
the investigation of earthquakes and the utility of controlled explosions in
search for oil and ore bodies \cite{fokkema1980reflection,
fokkema1977elastodynamic, sherwood1958elastic}. The unbounded rough
surface scattering problems are quite challenging due to the unbounded surfaces.
The usual Sommerfeld (for acoustic waves) or Silver--M\"{u}ller (for
electromagnetic waves) radiation condition is not valid any more
\cite{Arens2005, Zhang1998}. The Fredholm alternative theorem is not applicable
either due to the lack of compactness result. For the time-harmonic problems, we
refer to \cite{Chandler2006, Chandler2005, Chandler1999, Lechleiter2010,
LiShen2012} for some mathematical studies on the two-dimensional Helmholtz
equation and \cite{Haddar2011, LiWuZheng2011, LiZhengZheng2016} for the
three-dimensional Maxwell equations. The time-domain scattering problems have
recently attracted considerable attention due to their capability
of capturing wide-band signals and modeling more general material and
nonlinearity \cite{ChenMonk2014, jin2009finite, LiHuang2013, Riley2008,
Wang2012}, which motivates us to tune our focus from seeking the best possible
conditions for those physical parameters to the time-domain problem. Comparing
with the time-harmonic problems, the time-domain problems are less
studied due to the additional challenge of the temporal dependence.
The analysis can be found in \cite{Chen2008Maxwell, wang2014} for
the time-domain acoustic and electromagnetic obstacle scattering problems.
We refer to \cite{LiWangWood2015} and \cite{GaoLi2016} for the analysis of the
time-dependent electromagnetic scattering from an open cavity and a periodic
structure, respectively.

The acoustic-elastic interaction problems have received much attention in both
the mathematical and engineering communities
\cite{dallas1989, donea1982arbitrary, hamdi1986mixed, Hsiao1989, Hsiao1994,
LukeMartin1995}. There are also some numerical studies on the inverse problems
arising from the fluid-solid interaction such as reconstruction of surfaces of
periodic structures or obstacles \cite{HuKrish2016, TaoHu2016}. Many
approaches have been attempted to solve numerically the time-domain problems
such as coupling of boundary element and finite element with different time
quadratures \cite{estorff1991, soares2006dynamic, Flemisch2006}. However, the
rigorous mathematical study is still open at present.

In this work, we intend to answer the mathematical questions on well-posedness
and stability of the time-domain acoustic-elastic interaction problem in an
unbounded structure. The problem is reformulated as an initial-boundary value
problem by adopting an exact transparent boundary condition (TBC). Using the
Laplace transform and energy method, we show that the reduced variational
problem has a unique weak solution in the frequency domain. Meanwhile, we obtain
the stability estimate to show the existence of the solution in the
time-domain. In addition, we achieve a priori estimates with explicit
dependence on the time for the pressure of the acoustic wave and
the displacement of the elastic wave by considering directly the time-domain
variational problem and taking special test functions.

The paper is organized as follows. In section \ref{PF}, we introduce the model
equations and interface conditions for the acoustic-elastic interaction problem.
The time-domain TBC is presented and some trace results are proved. Section
\ref{srp} is devoted to the analysis of the reduced problem, where the
well-posdeness and stability are addressed in both the frequency and time
domains. We conclude the paper with some remarks  in section \ref{cl}.

\section{Problem formulation}\label{PF}

In this section, we define some notation, introduce the model equations, and
present an initial-boundary value problem for the acoustic-elastic
scattering in an air/fluid-solid medium.

\subsection {Problem Geometry}

\begin{figure}
\centering
\includegraphics[width=0.5\textwidth]{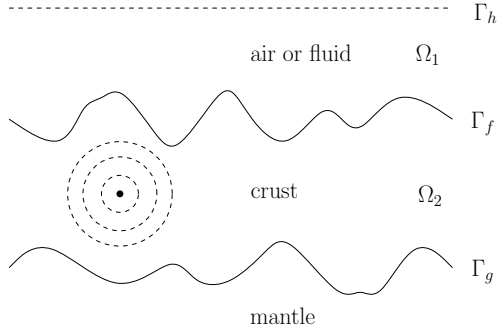}
\caption{Problem geometry of the acoustic-elastic interaction in an unbounded
structure.}
\label{pg}
\end{figure}

As shown in Figure \ref{pg}, we consider an active source which is embedded in
an elastic solid medium. It models an earthquake focus located in the crust
which lies between the lithosphere and the rigid mantle of the Earth. Due to the
excitation of the source, an elastic wave is generated in the solid and
propagates through to the medium of the air/fluid. Clearly, this process leads
to the air/fluid-solid interaction problem with the scattering interface
separating the domains where the acoustic and elastic waves travel.

Let $\boldsymbol{r}=(x, y)^\top\in\mathbb{R}^2$ and
$\boldsymbol{x}=(x, y, z)^\top\in\mathbb{R}^3$. Denote by $\Gamma_f=
\{\boldsymbol x\in \mathbb R^3: z= f(\boldsymbol r)\}$ the surface separating
the air/fluid and the solid, where $f$ is assumed to be a $W^{1, \infty}(\mathbb
R^2)$ function. Let $\Gamma_g =\{\boldsymbol x\in \mathbb R^3: z = g(\boldsymbol
r)\}$ be the surface separating the crust and the mantle, where $g$ is a
$L^\infty(\mathbb R^2)$ function satisfying $g (\boldsymbol r) <
f(\boldsymbol r), \boldsymbol{r}\in\mathbb R^2$. We assume that the open space
$\Omega^{+}_f =\{ \boldsymbol x \in \mathbb R^3: z >  f(\boldsymbol r)\}$ is
filled with a homogeneous compressible air or a compressible inviscid fluid with
the constant density $\rho_1$. The space $ \Omega_2 =\left \{ \boldsymbol x \in
\mathbb R^3:  g (\boldsymbol r) \leq  z \leq f (\boldsymbol r) \right \}$
is assumed to be occupied by a homogeneous isotropic linear elastic solid
which is characterized by the constant mass density $\rho_2$ and Lam\'{e}
parameters $\mu, \lambda$. Define an artificial planar surface
$\Gamma_h=\{\boldsymbol x\in\mathbb{R}^3, z=h\}$, where
$h>\sup_{\boldsymbol r\in\mathbb{R}^2}f(\boldsymbol r)$ is a constant.
Let $\Omega_1= \left\{ \boldsymbol x \in \mathbb R^3: f (\boldsymbol r)< z <h
\right\}$ and $\Omega=\Omega_1 \cup \Omega_2.$

\subsection{Acoustic wave  equation}

The acoustic wave field in air/fluid is governed by the conservation and the
dynamics equations in the  time-domain:
\begin{align}\label{cde}
\nabla p (\boldsymbol x, t)= -\rho_1 \partial_t \boldsymbol v  (\boldsymbol
x, t),\quad  c^2 \rho_1 \nabla \cdot \boldsymbol v (\boldsymbol x, t) =
-\partial_t p (\boldsymbol x, t), \quad \boldsymbol x \in \Omega^{+}_f,\, t>0,
\end{align}
where $p$ is the pressure, $\boldsymbol v$ is the velocity, and the
constants $\rho_1 >0 $ and $c>0$ are the density  and sound speed,
respectively. Eliminating the velocity  $\boldsymbol v$ from
\eqref{cde}, we obtain the acoustic wave equation for the pressure $p$:
\begin{align*}
 \Delta p (\boldsymbol x, t)- \frac{1}{c^2} \partial_t^2 p (\boldsymbol x, t)=0,
\quad  \boldsymbol x \in \Omega^{+}_f,\, t>0.
\end{align*}
The equation is constrained by the homogeneous initial conditions:
\[
p|_{t=0}=0, \quad \partial_t  p|_{t=0}=0, \quad \boldsymbol x\in\Omega^{+}_f.
\]
It follows from the conservation equation in \eqref{cde} that $\nabla \times
\boldsymbol v (\boldsymbol x, t)=0,$ i.e., the acoustic air/fluid
is irrotational. Thus there exists a scalar potential function $\varphi$ such
that $\boldsymbol v (\boldsymbol x, t) = \nabla \varphi (\boldsymbol x).$ It is
easy to note from \eqref{cde} that the corresponding dynamic component of
the pressure is given by
 \[
  p (\boldsymbol x, t) =- \rho_1 \partial_t \varphi (\boldsymbol x, t).
\]

\subsection{Elastic wave equation}

For the solid, the elastic wave field  in a homogeneous isotropic solid material
satisfies the linear time-domain  elasticity equation:
\begin{align}\label{ew}
 \nabla \cdot \boldsymbol  \sigma (\boldsymbol u (\boldsymbol x, t))
 -\rho_2 \partial_t^2 \boldsymbol u (\boldsymbol x, t)= \boldsymbol j
(\boldsymbol x, t), \quad \boldsymbol x \in \Omega_2,\,t>0,
\end{align}
where $\boldsymbol u =(u_1, u_2, u_3)^\top$ is the displacement vector,
$\rho_2>0$ is the density of the elastic solid material, $\boldsymbol j$ is the
source which models the earthquake focus and is assumed to have a compact
support contained in $\Omega_2$, and the symmetric stress tensor $\boldsymbol
\sigma (\boldsymbol u)$ is given by the
generalized Hook law:
\begin{equation}\label{sig}
 \boldsymbol \sigma (\boldsymbol u) = 2 \mu {\mathcal E} (\boldsymbol u)+\lambda
 {\rm tr} \left({\mathcal E} (\boldsymbol u) \right) {\rm  I},
 \quad {\mathcal E} (\boldsymbol u) =\frac{1}{2} \big(\nabla \boldsymbol u
+(\nabla \boldsymbol u)^{\top} \big).
\end{equation}
Here $\mu, \lambda$ are the Lam\'{e} parameters satisfying $\mu>0, \lambda +\mu
>0$, $\rm I \in \mathbb R^{3\times 3}$ is the identity matrix,
${\rm tr}(\mathcal{E}(\boldsymbol u))$ is the trace of the matrix
$\mathcal{E}(\boldsymbol u)$, and $\nabla \boldsymbol u$ is the displacement
gradient tensor given by
\begin{align*}
 \nabla \boldsymbol u=
 \left[
 \begin{matrix}
  \partial_x u_1 & \partial_y u_1 & \partial_z u_1 \\
  \partial_x u_2 & \partial_y u_2 & \partial_z u_2 \\
  \partial_x u_3& \partial_y u_3 & \partial_z u_3
 \end{matrix}
\right].
\end{align*}
Substituting \eqref{sig} into \eqref{ew}, we obtain the time-domain Navier
equation for the displacement $\boldsymbol u$:
 \begin{equation}\label{ne}
  \mu \Delta \boldsymbol u (\boldsymbol x, t)+ (\lambda + \mu ) \nabla \nabla
\cdot \boldsymbol u (\boldsymbol x, t) -\rho_2 \partial_t^2 \boldsymbol u
(\boldsymbol x, t) = \boldsymbol  j (\boldsymbol x,t),
  \quad \boldsymbol x \in \Omega_2,\, t>0.
 \end{equation}
By assuming that the mantle is rigid, we have
\[
 \boldsymbol u=0 \quad \text{on}~\Gamma_g,\, t>0.
\]
The elastic wave equation \eqref{ne} is constrained by the homogeneous initial
conditions:
\begin{align*}
 \boldsymbol u|_{t=0}=0, \quad \partial_t  \boldsymbol u|_{t=0}=0,
\quad  \boldsymbol x \in \Omega_2.
\end{align*}

\subsection{Interface conditions}

To couple the acoustic wave equation in the air/fluid and the elastic wave
equation in the solid, the kinematic interface condition is imposed to ensure
the continuity of the normal component of the
velocity on $\Gamma_f$:
\begin{equation}\label{kic}
 \boldsymbol n \cdot \boldsymbol v (\boldsymbol x, t) =\boldsymbol n \cdot
\partial_t  \boldsymbol u  (\boldsymbol x, t), \quad \boldsymbol x \in
\Gamma_f,\,t>0,
\end{equation}
where $\boldsymbol n$ is the unit normal on $\Gamma_f$
pointing from $\Omega_2$ to  $\Omega_1$. Noting $\boldsymbol v (\boldsymbol x,
t) =\nabla \varphi (\boldsymbol x, t)$ and $p (\boldsymbol x, t)=-\rho_1
\partial_t \varphi (\boldsymbol x, t)$, we have from \eqref{kic} that
\[
 \partial_{\boldsymbol n} p= \boldsymbol n \cdot \nabla p
=-\rho_1 \boldsymbol n \cdot \partial_t^2 \boldsymbol u
 \quad \text {on}~\Gamma_f,\, t>0.
\]
In addition, the dynamic interface condition is required to ensure the
continuity of traction:
\[
- p \boldsymbol n = \boldsymbol \sigma (\boldsymbol u)\cdot \boldsymbol  n,
\quad\boldsymbol x \in \Gamma_f,\, t>0,
\]
where $\boldsymbol \sigma (\boldsymbol u) \cdot \boldsymbol n$ is denoted as
the multiplication of the stress tensor $\boldsymbol \sigma (\boldsymbol u)$
with the normal vector $\boldsymbol n.$

\subsection{Laplace transform and some functional spaces}

We first  introduce some properties of the Laplace transform. For any
$s=s_1+{\rm i} s_2$ with $s_1>0, s_2\in\mathbb{R}$, define $\breve
{\boldsymbol u}(s)$ to be the Laplace transform of the function $\boldsymbol
u(t)$, i.e.,
\[
\breve {\boldsymbol u}(s) =\mathscr L (\boldsymbol u) (s)= \int_0^{\infty} e^{-s
t} \boldsymbol u (t) {\rm d} t.
\]
It follows from the integration by parts that
\begin{align*}
\int_0^t \boldsymbol u (\tau) {\rm d} \tau =\mathscr L^{-1} (s^{-1} \breve
{\boldsymbol u} (s)),
\end{align*}
where $\mathscr L^{-1}$ is the inverse Laplace transform. It is easy to verify
from the inverse Laplace transform that
\begin{equation*}
\boldsymbol u(t)=\mathscr F^{-1} \big( e^{s_1 t} \mathscr L (\boldsymbol u)
(s_1+s_2)\big),
\end{equation*}
where $\mathscr F ^{-1}$ denotes the inverse Fourier transform with respect to
$s_2.$

Recall the Plancherel or Parseval identity for the Laplace transform (cf.
\cite[(2.46)]{Cohen2007}):
 \begin{equation}\label{PI}
 \frac{1}{2 \pi} \int_{- \infty}^{\infty} \breve {\boldsymbol  u} (s)\cdot
\breve{\boldsymbol   v} (s) {\rm d} s_2= \int_0^{\infty} e^{- 2 s_1 t}
  {\boldsymbol u} (t)\cdot{\boldsymbol v} (t) {\rm d} t,\quad\forall ~
s_1>\lambda,
 \end{equation}
where $\breve  {\boldsymbol u}= \mathscr L (\boldsymbol u), \breve {\boldsymbol
v}= \mathscr L (\boldsymbol v)$ and $\lambda$ is the abscissa of convergence for
the Laplace transform of $\boldsymbol u$ and $\boldsymbol v$.

The following lemma (cf. \cite[Theorem 43.1]{Treves1975}) is an analogue of
Paley--Wiener--Schwarz theorem for the Fourier transform of the distributions
with compact supports in the case of the Laplace transform.

\begin{lemma}\label {A2}
Let $\breve {\boldsymbol h} (s)$ be a holomorphic function in the half-plane
$s_1 > \sigma_0$ and be valued in the Banach space $\mathbb E$. The following
two conditions are equivalent:
\begin{enumerate}

\item there is a distribution $\boldsymbol h \in \mathcal D_{+}'(\mathbb E)$
whose Laplace transform is equal to $\breve{\boldsymbol h}(s)$;

\item there is a real $\sigma_1$ with $\sigma_0 \leq \sigma_1 <\infty$ and an
integer $m \geq 0$ such that for all complex numbers $s$ with ${\rm Re} s =s_1 >
\sigma_1,$ we have $\| \breve {\boldsymbol  h} (s)\|_{\mathbb E} \lesssim
(1+|s|)^{m}$,

\end{enumerate}
where $\mathcal D'_{+}(\mathbb E)$ is the space of distributions on the real
line which vanish identically in the open negative half line.
\end{lemma}

Next we introduce some Sobolev spaces. For any $u(\cdot, h) \in L^2(\Gamma_h)$
which is identified as $L^2(\mathbb{R}^2)$, we denote by $\hat{u}(\boldsymbol\xi, h)$ 
the Fourier transform of $u (\boldsymbol r, h)$:
\[
\hat u (\boldsymbol \xi, h)=\frac{1}{2 \pi}\int_{\mathbb R^2} u (\boldsymbol r,
h) e^{-{\rm i} \boldsymbol r \cdot \boldsymbol  \xi} {\rm d}
\boldsymbol r,
\]
where $\boldsymbol \xi = (\xi_1, \xi_2)^\top \in \mathbb R^2.$ For any
$\alpha\in\mathbb{R}$, define the functional space
 \[
 H^\alpha(\Gamma_h)=\left\{ u (\boldsymbol r, h) \in L^2
(\mathbb R^2):  \int_{\mathbb R^2} (1+ |\boldsymbol \xi|^2)^{\alpha} |\hat
u(\boldsymbol\xi, h)|^2{\rm d}\boldsymbol  \xi < \infty \right\},
 \]
which is a Sobolev space under the norm
\[
\|u\|_{H^{\alpha} (\Gamma_h)} =\left[ \int_{\mathbb R^2} (1+ |\boldsymbol
\xi|^2)^{\alpha} |\hat u(\boldsymbol\xi, h)|^2 {\rm d} \boldsymbol  \xi
\right]^{1/2}.
\]
It is clear to note that the dual space associated with $H^{\alpha} (\Gamma_h)$
is the space $H^{- \alpha} (\Gamma_h)$ with respect to the scalar
product in $L^2 (\mathbb R^2)$ defined by
\[
\langle u,  v \rangle_{\Gamma_h} =\int_{\Gamma_h}u(\boldsymbol r,
h)\bar{v}(\boldsymbol r, h){\rm d}\boldsymbol{r}=\int_{\mathbb R^2} \hat
u(\boldsymbol\xi, h) \bar {\hat v}(\boldsymbol\xi, h) {\rm d }\boldsymbol\xi.
\]
Denote by $H^{1/2}(\Gamma_f)$ the Sobolev trace space, the subspace of
$L^2(\Gamma_f)$ such that
\[
 \int_{\Gamma_f}|u(\boldsymbol r, f(\boldsymbol r))|^2{\rm d}\boldsymbol r
 +\int_{\Gamma_f}\int_{\Gamma_f}\frac{|u(\boldsymbol r_1,
f(\boldsymbol r_1))-u(\boldsymbol r_2, f(\boldsymbol r_2))|^2}{|\boldsymbol r_1
-\boldsymbol r_2|^3}{\rm d}\boldsymbol r_1{\rm d}\boldsymbol r_2<\infty.
\]
$H^{1/2}(\Gamma_f)$ is equipped with the norm
\[
 \|u\|_{H^{1/2}(\Gamma_f)}=\left( \int_{\Gamma_f}|u(\boldsymbol r, f(\boldsymbol
r))|^2{\rm d}\boldsymbol r  +\int_{\Gamma_f}\int_{\Gamma_f}\frac{|u(\boldsymbol
r_1, f(\boldsymbol r_1))-u(\boldsymbol r_2, f(\boldsymbol r_2))|^2}{|\boldsymbol
r_1 -\boldsymbol r_2|^3}{\rm d}\boldsymbol r_1{\rm d}\boldsymbol r_2
\right)^{1/2}.
\]

Denote by $H^{\nu} (\Omega)= \{D^{\alpha} u \in L^2 (\Omega)~\text {for all}~
|\alpha| \leq \nu\}$ the standard Sobolev space of square integrable
functions with the order of derivatives up to $\nu$. Let $H^1_{\Gamma_g}
(\Omega)=\{ u\in H^1 (\Omega): u =0 ~\text {on}~\Gamma_g\}$. Let $H^1_{\Gamma_g}
(\Omega)^3$ and $H^{1/2} (\Gamma_f)^3$ be the Cartesian product spaces
equipped with the corresponding 2-norms of $H^1_{\Gamma_g} (\Omega)$ and
$H^{1/2} (\Gamma_f)$, respectively. For any $\boldsymbol u=(u_1, u_2,
u_3)^\top \in H^1_{\Gamma_g} (\Omega_2)^3$, define the Frobenius norm:
\[
\|\nabla {\boldsymbol u}\|_{F (\Omega_2)}= \left ( \sum \limits_{j=1}^3
\int_{\Omega_2}  |\nabla u_j |^2  {\rm d}  \boldsymbol  x  \right)^{1/2}.
\]
It is easy to verify that
\begin{align}\label{f1}
 \|\nabla \boldsymbol u\|^2_{F (\Omega_2)} +\|\nabla \cdot \boldsymbol
u\|^2_{L^2 (\Omega_2)} \lesssim \|\boldsymbol u\|^2_{H^1 (\Omega_2)^3}.
\end{align}
Hereafter, the expression $a \lesssim  b$ or $a\gtrsim c$ stands for $a \leq  C
b$ or $a\geq C b$, where $C$ is a positive constant and its specific value is
not required but should be always clear from the context.

\subsection{Transparent boundary condition}

In this subsection, we will introduce an exact time-domain TBC to formulate the
acoustic-elastic wave interaction problem into the following coupled
initial-boundary value problem:
\begin{equation}\label{rp}
 \begin{cases}
  \Delta p  - \frac{1}{c^2 } \partial_t^2  p =0 \quad  &\text{in}~
\Omega_1,\,t>0\\
  \mu \Delta \boldsymbol u +  (\lambda+\mu) \nabla \nabla \cdot \boldsymbol u
-\rho_2 \partial_t^2 \boldsymbol u =\boldsymbol j
  \quad  & \text {in}~  \Omega_2 ,\,t>0,\\
  p|_{t=0}=\partial_t p|_{t=0}=0,\quad \boldsymbol u|_{t=0}=\partial_{t}
\boldsymbol u|_{t=0}=0 \quad & \text {in}~\Omega,\\
\partial_{\boldsymbol n} p=-\rho_1 \boldsymbol n \cdot
\partial_t^2 \boldsymbol u, \quad -p \boldsymbol n = \boldsymbol \sigma
(\boldsymbol u) \cdot \boldsymbol  n\quad & \text {on}~\Gamma_f,\,t>0,\\
  \partial_{\boldsymbol \nu} p= \mathscr T p, \quad  &\text
{on}~~\Gamma_h,\,t>0,\\
  \boldsymbol u =0 \quad & \text {on}~ \Gamma_g,\,t>0,
 \end{cases}
 \end{equation}
where $ \boldsymbol  \nu=(0, 0, 1)^\top$ is the unit normal vector on $\Gamma_h$
pointing from $\Omega_1$ to $\Omega_h^+=\{\boldsymbol x\in\mathbb R^2: z>h\}$,
and $\mathscr T$ is the time-domain TBC operator on $\Gamma_h$. In
what follows, we shall derive the formulation of the operator $\mathscr T$ and
show some of its properties.

Let $\breve p (\boldsymbol x, s)  = \mathscr L (p)$ and $\breve {\boldsymbol u}
(\boldsymbol x, s)=\mathscr L (\boldsymbol u)$ be the Laplace transform of
$p(\boldsymbol x, t)$ and $\boldsymbol u (\boldsymbol x, t)$ with respect to
$t$, respectively. Recall that
\begin{align*}
 &\mathscr L (\partial_t p) = s \breve p (\cdot, s )- p (\cdot, 0),\quad
 \mathscr  L (\partial_t^2 p)= s^2 \breve {p} (\cdot, s)- s p (\cdot,
0)-\partial_t p (\cdot, 0),\\
& \mathscr L (\partial_t \boldsymbol u)= s \breve {\boldsymbol u} (\cdot, s)-
\boldsymbol u (\cdot, 0),
 \quad \mathscr L (\partial_t^2  \boldsymbol u) = s^2 \breve {\boldsymbol u }
(\cdot, s) - s \boldsymbol u (\cdot, 0)- \partial_t \boldsymbol u (\cdot,0).
\end{align*}
Taking the Laplace transform of \eqref{rp} and using the initial conditions, we
obtain the acoustic-elastic wave interaction problem in the $s$-domain:
\[
 \begin{cases}
  \Delta  \breve {p}  - \frac{s^2}{c^2 } \breve   p =0 \quad  &\text{in}~
\Omega_1,\\
  \mu \Delta \breve { \boldsymbol u } +  (\lambda+\mu) \nabla \nabla \cdot
\breve { \boldsymbol u} -\rho_2 s^2  \breve {\boldsymbol u} = \breve
{\boldsymbol j}\quad  & \text {in}~  \Omega_2,\\
  \partial_{\boldsymbol n}\breve p=-\rho_1 s^2 \boldsymbol n
\cdot \breve {\boldsymbol u}, \quad - \breve {p} \boldsymbol n = \boldsymbol
\sigma ( \breve {\boldsymbol u}) \cdot  \boldsymbol n \quad & \text
{on}~\Gamma_f,\\
  \partial_{\boldsymbol \nu} \breve p= \mathscr{B}  {\breve p}, \quad  &\text
{on}~\Gamma_h, \\
  \breve {\boldsymbol u} =0 \quad & \text {on}~ \Gamma_g,
 \end{cases}
 \]
where $\breve {\boldsymbol  j}= \mathscr L (\boldsymbol j)$, $\mathscr{B}$ is
the Dirichlet-to-Neumann (DtN) operator on $\Gamma_h$ in $s$-domain and
satisfies $\mathscr T= \mathscr L^{-1}\circ\mathscr B  \circ\mathscr L$.

In order to deduce the TBC, we consider the Helmholtz equation with a complex
wavenumber:
\begin{align}\label{pe}
 \Delta \breve {p} -\frac{s^2 }{c^2} \breve p=0 \quad \text {in}~\Omega_h^+.
\end{align}
Taking the Fourier transform of \eqref{pe} with respect to $\boldsymbol r$
yields
\begin{align}\label{so}
\begin{cases}
 \frac{{\rm d}^2 \hat {\breve p} (\boldsymbol \xi, z)}{ {\rm d} z^2}- \left(
\frac{s^2 }{c^2 } + |\boldsymbol \xi|^2\right) \hat {\breve p } (\boldsymbol\xi,
z)=0, \quad &  z>h,\\
 \hat {\breve p} ( \boldsymbol \xi, z) =\hat {\breve p} (\boldsymbol\xi, h),
\quad &z=h.
 \end{cases}
\end{align}
Solving \eqref{so} and using the bounded outgoing wave condition, we get
\[
 \hat {\breve p} (\boldsymbol\xi, z)= \hat{\breve p} (\boldsymbol\xi, h) e^{-
\beta (\boldsymbol\xi) (z-h)}, \quad z >h,
\]
where
\begin{align}\label{B1}
 \beta^2(\boldsymbol \xi) = \frac{s^2 }{c^2 } +|\boldsymbol \xi|^2 \quad\text
{with}~{\rm Re}\beta (\boldsymbol \xi) >0.
\end{align}
Thus we obtain the solution of \eqref{pe}:
\begin{align}\label{sol}
 \breve p (\boldsymbol r, z) = \int_{\mathbb R^2}
 \hat{\breve p} (\boldsymbol \xi,h) e^{- \beta ( \boldsymbol \xi) (z-h)} e^{
{\rm i}\boldsymbol \xi \cdot \boldsymbol r} {\rm d}  \boldsymbol\xi.
\end{align}
Taking the normal derivative of \eqref{sol} on $\Gamma_h$ and evaluating it at
$z=h$, we have
\begin{align*}
 \partial_{\boldsymbol \nu} \breve p (\boldsymbol r, h) =\int_{\mathbb R^2} -
\beta ( \boldsymbol \xi ) \hat  {\breve p}( \boldsymbol \xi, h) e^{{\rm i}
\boldsymbol \xi \cdot  \boldsymbol r} {\rm d } \boldsymbol\xi.
\end{align*}

For any function $u (\boldsymbol r, h)$ defined on $\Gamma_h$, we  defined
the DtN operator
\begin{align}\label{dtn}
 \left(\mathscr{B}  u \right) (\boldsymbol r, h)= \int_{\mathbb R^2} - \beta
(\boldsymbol \xi) \hat  u ( \boldsymbol \xi, h) e^{{\rm i} \boldsymbol \xi \cdot
\boldsymbol r } {\rm d} \boldsymbol \xi.
\end{align}

Let $z_1, z_2$ be two constants satisfying $z_2<z_1$. Define
$\Gamma_j=\{\boldsymbol x\in\mathbb R^2: z=z_j\}$ and $R=\{\boldsymbol
x\in\mathbb{R}^3: \boldsymbol r\in\mathbb R^2, z_2<z<z_1\}$. The following
several trace results are useful in subsequent analysis.

\begin{lemma}\label{tr}
Let $\gamma_0=(1+(z_1 -z_2)^{-1})^{1/2}$. We have the estimate
\[
\|u\|_{H^{1/2}(\Gamma_j)}\leq \gamma_0 \| u\|_{H^1(R)},\quad\forall u\in
H^1(R).
\]
\end{lemma}

\begin{proof}
First we have
\begin{align*}
(z_1 -z_2)|\zeta(z_j)|^2&=\int_{z_2}^{z_1}|\zeta(z)|^2
{\rm d}z+\int_{z_2}^{z_1}\int_z^{z_j}\frac{\rm d}{{\rm d}\tau}|\zeta(\tau)|^2
{\rm d}\tau{\rm d}z\\
&\leq\int_{z_2}^{z_1}|\zeta(z)|^2{\rm
d}z+(z_1 -z_2)\int_{z_2}^{z_1}2|\zeta(z)||\zeta'(z)|{\rm d}z,
\end{align*}
which implies by the Cauchy--Schwarz inequality that
\begin{equation}\label{cs1}
(1+|\boldsymbol\xi|^2)^{1/2}|\zeta(z_j)|^2\leq \gamma_0^2
(1+|\boldsymbol\xi|^2)\int_{z_2}^{z_1}|\zeta(z)|^2
{\rm d}z+\int_{z_2}^{z_1}|\zeta'(z)|^2 {\rm d}z.
\end{equation}

Given $u$ in $H^1(R)$, a simple calculation yields that
\begin{equation}\label{cs2}
\|u\|^2_{H^{1/2}(\Gamma_j)}=\int_{\mathbb{R}^2}(1+|\boldsymbol\xi|^2)^{1/2}
|\hat{u} (\boldsymbol\xi, z_j)|^2 {\rm d}\boldsymbol\xi
\end{equation}
and
\begin{equation}\label{cs3}
\|u\|^2_{H^1(R)}=\int_{z_2}^{z_1}\int_{\mathbb{R}^2}\left[
\left(1+|\boldsymbol\xi|^2\right)|\hat{u}(\boldsymbol\xi,
z)|^2+|\hat{u}'(\boldsymbol\xi, z)|^2\right] {\rm d}\boldsymbol\xi {\rm d}z,
\end{equation}
where $\hat{u}'(\boldsymbol\xi, z)=\partial_z \hat{u}(\boldsymbol\xi, z)$.

Using \eqref{cs1}, we obtain
\begin{align*}
(1+|\boldsymbol\xi|^2)^{1/2}|\hat{u}(\boldsymbol\xi, z_j)|^2 &\leq\gamma_0^2
(1+|\boldsymbol\xi|^2)\int_{z_2}^{z_1}|\hat{u}(\boldsymbol\xi, z)|^2
{\rm d}z+\int_{z_2}^{z_1}|\hat{u}'(\boldsymbol\xi, z)|^2 {\rm d}z\\
&\leq\gamma_0^2\int_{z_2}^{z_1}\left[(1+|\boldsymbol\xi|^2)|\hat{u}
(\boldsymbol\xi, z)|^2+|\hat{u}'(\boldsymbol\xi, z)|^2\right]{\rm d}z,
\end{align*}
which completes the proof after combining \eqref{cs2} and \eqref{cs3}.
\end{proof}

\begin{lemma}\label{trf}
 There exists a positive constant $C$ such that
 \[
  \|u\|_{H^{1/2}(\Gamma_f)}\leq C\|u\|_{H^1(\Omega_1)},\quad\forall u\in
H^1(\Omega_1).
 \]
\end{lemma}

\begin{proof}
Consider the change of variables:
\[
  \tilde x =x, \quad  \tilde y = y, \quad
  \tilde z= h \left(  \frac{z-f}{h-f}\right),
\]
which maps the domain $\Omega_1$ into the rectangular slab $D_1:=
\{\tilde{\boldsymbol x}=(\tilde x, \tilde y, \tilde z ) \in \mathbb R^3: 0<
\tilde z < h\}$. In particular, the surface $\Gamma_f$ is transformed to the
planar surface $\Gamma_0:=\left\{\tilde{\boldsymbol x}\in \mathbb R^3:
\tilde z =0 \right\}$. Let $J$ be the Jacobian matrix of the transformation. A
simple calculation yields that
\begin{align*}
|J|=\left|\frac{\partial(\tilde x, \tilde y, \tilde z) }{\partial (x, y,
z)}\right|=\left|
 \begin{matrix}
  1 & 0 & 0\\
  0 & 1&0\\
  \frac{h(z -h) \partial_x f}{(h -f)^2} &  \frac{h(z -h) \partial_{y} f}{(h
-f)^2} &\frac{h}{h-f}
 \end{matrix}
\right|=\frac{h}{h-f} \neq 0,
\end{align*}
which shows that the transformation is invertible. Denote by $J^{-1}$ the
inverse of the Jacobian matrix. It follows from Lemma \ref{tr} that we have
\begin{align}\label{l1}
 \|u \|_{H^{1/2} (\Gamma_0) } \lesssim \|u \|_{H^1(D_1)}.
\end{align}

Using the usual Sobolev norm in $\Omega_1$ and the change of variables, we get
\begin{align}\label{l2}
 \|u\|^2_{H^1 (\Omega_1)}
 =&\int_{\Omega_1} \left( |u (\boldsymbol x)|^2 + | \nabla u (\boldsymbol x)|^2
\right) {\rm d} \boldsymbol x \nonumber\\
 =&\int_{D_1} \bigg[ |u|^2 + \left|\partial_{\tilde x} u - \partial_{\tilde x}
f\Big(\frac {h-\tilde z}{h-f}\Big)  \partial_{\tilde z} u \right|^2 \nonumber\\
&\qquad+\left|\partial_{\tilde y}u - \partial_{\tilde y} f \Big(\frac{h - \tilde
z}{h-f}\Big) \partial_{\tilde z} u\right|^2+ \left|\Big(\frac{h}{h-f}\Big)
\partial_{\tilde z} u \right|^2 \bigg] J^{-1} {\rm d} \tilde{\boldsymbol x}
\nonumber\\
 \lesssim &  \int_{D_1} \left( |u (\tilde{\boldsymbol x})|^2 + |\nabla u
(\tilde{\boldsymbol x})|^2  \right){\rm d} \tilde{\boldsymbol x}=\|u\|^2_{H^1
(D_1)},
\end{align}
where we have used the assumption that $f\in W^{1, \infty}(\mathbb{R}^2)$. On
the other hand, we have
\begin{align}\label{l3}
 \|u \|^2_{H^1 (D_1)}
=& \int_{D_1} \left ( |u (\tilde{\boldsymbol x})|^2  + |\nabla u
(\tilde{\boldsymbol x})|^2 \right) {\rm d} \tilde{\boldsymbol x} \nonumber\\
 =& \int_{\Omega_1 } \bigg[ |u|^2  + \left|\partial_x u
+\partial_x f\Big(\frac{h-z}{h-f}\Big) \partial_z u \right|^2\nonumber\\
&\qquad +\left| \partial_y u +\partial_y f\Big(\frac{h-z}{h-f}\Big) \partial_z u
\right|^2+ \left(\partial_z u \frac{h-f}{h} \right)^2\bigg]
 J{\rm d}\boldsymbol x \nonumber \\
  \lesssim & \int_{\Omega_1} \left(  |u (\boldsymbol x)|^2 + |\nabla
u(\boldsymbol x)|^2\right) { \rm d}\boldsymbol x=\|u\|^2_{H^1(\Omega_1)}.
\end{align}
Combining \eqref{l2} and \eqref{l3}, we get that the norm
$\|u\|^2_{H^1(\Omega_1)}$ is equivalent to the norm $\|u\|^2_{H^1(D_1)}$.

Next, we prove the equivalence of the norm $\|u\|_{H^{1/2} (\Gamma_0)}$ and
the norm $\|u\|_{H^{1/2}(\Gamma_f)}$. First we have
\begin{align*}
 \| u\|^2_{H^{1/2} (\Gamma_0)}
 &= \int_{\Gamma_0} | u (\tilde{\boldsymbol r}, 0)|^2 {\rm d
}\tilde{\boldsymbol r}+\int_{\Gamma_0}\int_{\Gamma_0} \frac{ |u
(\tilde{\boldsymbol r}_1, 0) - u(\tilde{\boldsymbol r}_2,
0)|^2}{|\tilde{\boldsymbol r}_1 - \tilde{\boldsymbol
r}_2|^3}{\rm d}\tilde{\boldsymbol r}_1 {\rm d} \tilde{\boldsymbol r}_2.
\end{align*}
It follows from the change of variables that we have
\begin{align*}
\|u\|^2_{H^{1/2} (\Gamma_f)}=&
\int_{\Gamma_f} |u (\boldsymbol r, f (\boldsymbol r))|^2 {\rm d}\boldsymbol
r+\int_{\Gamma_f}\int_{\Gamma_f} \frac{ |u (\boldsymbol r_1, f (\boldsymbol
r_1)) - u(\boldsymbol r_2, f (\boldsymbol r_2))|^2}{|\boldsymbol r_1
- \boldsymbol r_2|^3}{\rm d} \boldsymbol r_1 {\rm d} \boldsymbol r_2\\
=&\int_{\Gamma_0}|u(\tilde{\boldsymbol r}, 0)|^2
(1+|\nabla_{\tilde{\boldsymbol r}} f|^2)^{1/2}{\rm
d}\tilde{\boldsymbol r}+\int_{\Gamma_0}\int_{\Gamma_0} \frac{ |u
(\tilde{\boldsymbol r}_1, 0) - u(\tilde{\boldsymbol r}_2,
0)|^2}{|\tilde{\boldsymbol r}_1 - \tilde{\boldsymbol
r}_2|^3}\\
&\times(1+|\nabla_{\tilde{\boldsymbol r}_1}f|^2)^{1/2}
(1+|\nabla_{\tilde{\boldsymbol r}_2}f|^2)^{1/2}
{\rm d}\tilde{\boldsymbol r}_1 {\rm d} \tilde{\boldsymbol r}_2.
\end{align*}
Hence we obtain
\[
 \|u\|_{H^{1/2}(\Gamma_0)}\leq \|u\|_{H^{1/2}(\Gamma_f)}\lesssim
\|u\|_{H^{1/2}(\Gamma_0)}.
\]
The proof is completed by using \eqref{l1} and the equivalence of the norms.
\end{proof}

\begin{lemma}\label{trg}
There exists a positive constant $C$ such that
\[
\|\boldsymbol u\|_{H^{1/2} (\Gamma_f)^3} \leq C\|\boldsymbol u\|_{H^1
(\Omega_2)^3}, \quad  \forall \boldsymbol u \in H^1_{\Gamma_g}(\Omega_2)^3.
\]
\end{lemma}

\begin{proof}
 Denote $\tilde{\Omega}_2=\{\boldsymbol{x}\in\mathbb{R}^3:
\boldsymbol{r}\in\mathbb{R}^2,\, \inf_{{\boldsymbol
r}\in\mathbb{R}^2}g(\boldsymbol{r})<z<f(\boldsymbol r)\}$ which contains the
domain $\Omega_2$. For any $\boldsymbol{u}\in H^1_{\Gamma_g}(\Omega_2)^3$, we
consider the zero extension to $\tilde{\Omega}_2$:
\begin{align*}
 \tilde{\boldsymbol u} (\boldsymbol x)=
 \begin{cases}
  \boldsymbol{u} (\boldsymbol x), \quad  &\boldsymbol x \in
\Omega_2,\\
  0, \quad  & \boldsymbol x \in \tilde{\Omega}_2 \setminus \bar{\Omega}_2.
  \end{cases}
\end{align*}
It is clear to note that
\begin{equation}\label{trg_s1}
 \|\boldsymbol{u}\|_{H^{1/2}(\Gamma_f)^3}=\|\tilde{\boldsymbol u}\|_{H^{1/2}
(\Gamma_f)^3},\quad \|\boldsymbol{u}\|_{H^1(\Omega_2)^3}=\|\tilde{\boldsymbol
u}\|_{H^1(\tilde{\Omega}_2)^3}.
\end{equation}
It follows from Lemmas \ref{tr} and \ref{trf} that there exists a positive
constant $C$ such that
\begin{equation}\label{trg_s2}
 \|\tilde{\boldsymbol u}\|_{H^{1/2} (\Gamma_f)^3}\leq C\|\tilde{\boldsymbol
u}\|_{H^1(\tilde{\Omega}_2)^3}.
\end{equation}
Combining \eqref{trg_s1} and \eqref{trg_s2} completes the proof.
\end{proof}

\begin{lemma}\label{cc}
 The DtN operator $\mathscr{B}: H^{1/2} (\Gamma_h) \to H^{-1/2}
(\Gamma_h)$ is continuous, i.e.,
 \begin{align*}
  \|\mathscr{B}u\|_{H^{-1/2} (\Gamma_h)} \lesssim \| u\|_{H^{1/2}
(\Gamma_h)},\quad\forall u\in H^{1/2} (\Gamma_h).
 \end{align*}
\end{lemma}

\begin{proof}
 For any $u \in H^{1/2} (\Gamma_h)$, it follows from \eqref{dtn} and
\eqref{B1} that
 \begin{align*}
  \|\mathscr{B} u\|^2_{H^{-1/2} (\Gamma_h)}
  =&\int_{\mathbb R^2} (1+|\boldsymbol \xi| ^2) ^{-1/2} | -\beta (\boldsymbol
\xi) \hat  u ( \boldsymbol \xi, h) |^2 {\rm d} \boldsymbol \xi\\
  = &\int_{\mathbb R^2} (1+|\boldsymbol \xi|^2)^{1/2}
  (1+|\boldsymbol \xi|^2)^{-1} |\beta ( \boldsymbol \xi)|^2 |\hat u
(\boldsymbol\xi, h)|^2 {\rm d} \boldsymbol  \xi
  \lesssim \|u\|^2_{H^{1/2} (\Gamma_h)},
 \end{align*}
where we have used
\[
 |\beta (\boldsymbol \xi)|^2 = \left|\frac{s^2}{c^2} + |\boldsymbol \xi|^2
\right| \leq \frac{|s|^2}{c^2} + |\boldsymbol \xi|^2 \lesssim 1+ |\boldsymbol
\xi|^2,
\]
which completes the proof.
\end{proof}

\begin{lemma}\label{tp}
 We have
 \[
  -{\rm Re} \langle s^{-1} \mathscr{B} u,  u \rangle_{\Gamma_h} \geq 0, \quad
\forall u \in H^{1/2}(\Gamma_h).
 \]
\end{lemma}

\begin{proof}
A simple calculation yields that
 \[
  - \langle s^{-1} \mathscr{B} u, u \rangle_{\Gamma_h}
  =    \int_{\mathbb R^2}  s^{-1} \beta (\boldsymbol\xi) |\hat u (\boldsymbol
\xi, h)|^2 {\rm d}\boldsymbol \xi= \int_{\mathbb R^2}  \frac{\bar s \beta
(\boldsymbol \xi)}{|s|^2}  |\hat u
(\boldsymbol\xi, h)|^2 {\rm d} \boldsymbol \xi
 \]
Let $\beta (\boldsymbol\xi) =a+{\rm i } b, s=s_1 +{\rm i} s_2$ with $a>0,
s_1>0$.
Taking the real part of the above equation gives
\begin{align}\label{i1}
- {\rm Re} \langle s^{-1}\mathscr{B} u, u \rangle_{\Gamma_h} = \int_{\mathbb
R^2} \frac{(s_1 a + s_2 b )}{|s|^2} |\hat u (\boldsymbol \xi, h)|^2 {\rm d}
\boldsymbol \xi.
\end{align}
Recalling $\beta^2 (\boldsymbol \xi) = \frac{s^2}{c^2 } + | \boldsymbol \xi|^2$,
we have
\begin{align}
  a^2-b^2 = \frac{s_1^2 -s_2^2 }{c^2} +|\boldsymbol  \xi|^2, \quad
  a b=\frac{s_1 s_2}{c^2}. \label{I2}
\end{align}
Substituting \eqref{I2} into \eqref{i1} yields
\begin{align*}
 - {\rm Re} \langle s^{-1}\mathscr{B} u, u \rangle_{\Gamma_h} =\int_{\mathbb
R^2} \frac{1}{|s|^2} \left( a s_1 + \frac{s_1}{a } \frac{s_2^2}{c^2} \right)
|\hat u ( \boldsymbol \xi, h)|^2 {\rm d}\boldsymbol  \xi \geq 0,
\end{align*}
which completes the proof.
\end{proof}

For any function $ u (\boldsymbol r, h)$ defined on $\Gamma_h$,
using the DtN operator \eqref{dtn}, we can obtain the following TBC in the
$s$-domain:
\begin{align}\label{tbc1}
 \partial_{\boldsymbol \nu } \breve p =\mathscr{B} \breve p \quad \text
{on}~\Gamma_h.
\end{align}
Taking the inverse Laplace transform of \eqref{tbc1} yields the TBC in the
time-domain:
\[
 \partial_{\boldsymbol \nu} p = \mathscr T p \quad  \text {on}~ \Gamma_h.
\]

\section{The reduced problem} \label{srp}

In this section, we present the main results of this paper, which include the
well-posedness and stability of the scattering problem and related a priori
estimates.

\subsection {Well-posedness in the $s$-domain}

Consider the reduced problem in the $s$-domain:
\begin{subequations}\label{saep}
\begin{numcases}{}
  \label{rp1} \Delta  \breve {p}  - \frac{s^2}{c^2 } \breve   p =0 & \quad
\text{in}~ $\Omega_1$,\\
  \label{rp2} \mu \Delta \breve { \boldsymbol u } +  (\lambda+\mu) \nabla \nabla
\cdot \breve { \boldsymbol u} -\rho_2 s^2  \breve {\boldsymbol u} = \breve
{\boldsymbol j} & \quad   \text {in}~ $\Omega_2$,\\
 \label{rp3}\partial_{\boldsymbol n}\breve{p}=-\rho_1 s^2 \boldsymbol n\cdot
\breve {\boldsymbol u}, \quad - \breve {p} \boldsymbol n = \boldsymbol \sigma (
\breve {\boldsymbol u}) \cdot \boldsymbol n & \quad \text{on}~$\Gamma_f$,\\
 \label{rp5}\partial_{\boldsymbol \nu} \breve p= \mathscr{B}  {\breve p}, &
\quad \text{on}~$\Gamma_h$,\\
  \label{rp6}\breve {\boldsymbol u} =0 & \quad  \text {on}~ $\Gamma_g$.
\end{numcases}
\end{subequations}
Multiplying \eqref{rp1} and \eqref{rp2} by the complex conjugate of a test
function $q \in H^1(\Omega_1)$ and a test function $\boldsymbol v \in
H^1_{\Gamma_g} (\Omega_2)^3 $, respectively, using the integration by parts and
boundary conditions, which include the TBC condition \eqref{rp5}, the
kinematic and dynamic interface conditions \eqref{rp3}, and the rigid
boundary condition \eqref{rp6}, we arrive at the variational problem: To find
$(\breve p, \breve {\boldsymbol u}) \in H^1(\Omega_1) \times H_{\Gamma_g}^1(\Omega_2)^3$  such that
\begin{equation} \label{iv1}
 \int_{\Omega_1} \left( \frac{1}{s}\nabla \breve p \cdot \nabla {\bar q}  +
\frac{s}{c^2} \breve p \bar q  \right) {\rm d} \boldsymbol  x-
 \langle  s^{-1}\mathscr{B} \breve p,  q \rangle_{\Gamma_h} - \rho_1  s
\int_{\Gamma_f}(\boldsymbol n \cdot \breve {\boldsymbol u}) \bar q {\rm } {\rm
d} {\gamma}=0
\end{equation}
and
\begin{align} \label{iv2}
 \int_{\Omega_2} & \frac{1}{s}\left(\left( \mu (\nabla {\breve {\boldsymbol
u}}: \nabla \bar {\boldsymbol v}) + (\lambda+\mu ) (\nabla \cdot \breve
{\boldsymbol u}) (\nabla \cdot \bar
{\boldsymbol v})\right) + \rho_2 s  \breve {\boldsymbol u} \cdot \bar
{\boldsymbol v} \right) {\rm d} \boldsymbol x \notag\\
& +\frac{1}{s} \int_{\Gamma_f} \breve  p (\boldsymbol n \cdot \bar {\boldsymbol
v}){\rm d} \gamma= - \int_{\Omega_2} \frac{1}{s} \breve {\boldsymbol j} \cdot
\bar {\boldsymbol v} {\rm d} \boldsymbol x,\quad\forall (q, \boldsymbol v)\in
H^1(\Omega_1)\times H^1_{\Gamma_g}(\Omega_2),
\end{align}
where $A:B= {\rm tr} \left( A B^{\top}\right)$ is the Frobenius inner product of
square matrices $A$ and $B$.

We multiply \eqref{iv2} by $\rho_1 |s|^2$ and add the obtained result to \eqref{iv1} to obtain an
equivalent variational problem: To find $(\breve p, \breve {\boldsymbol u}) \in
H^1(\Omega_1) \times H_{\Gamma_g}^1 (\Omega_2)^3$  such that
\begin{align}\label{ivp}
 a\left( \breve p, \breve {\boldsymbol u}; q, \boldsymbol v \right) =
 -\int_{\Omega_2}   \rho_1 \bar s   \breve  {\boldsymbol j} \cdot \bar
{\boldsymbol v} {\rm d} \boldsymbol  x, \quad \forall (q, \boldsymbol v) \in
H^1(\Omega_1) \times H_{\Gamma_g}^1
(\Omega_2)^3,
\end{align}
where the sesquilinear form
\begin{align}
 a\left( \breve p, \breve {\boldsymbol u}; q, \boldsymbol v \right)=
 &  \int_{\Omega_1} \left( \frac{1}{s}  \nabla \breve p \cdot \nabla {\bar q}  +
\frac{s}{c^2} \breve p \bar q  \right) {\rm d} \boldsymbol x + \int_{\Omega_2}
\Big( \rho_1 \bar s\big( \mu (\nabla {\breve {\boldsymbol
u}} : \nabla \bar {\boldsymbol v}) \notag\\
& + (\lambda+\mu ) (\nabla \cdot \breve {\boldsymbol u}) (\nabla \cdot \bar
{\boldsymbol v})\big) + \rho_1 \rho_2 s |s|^2  \breve {\boldsymbol u} \cdot \bar
{\boldsymbol v}  \Big) {\rm d} \boldsymbol  x -\langle s^{-1} \mathscr{B}\breve
p, q \rangle_{\Gamma_h}\notag\\
&+\rho_1 \int_{\Gamma_f} \left( \bar s  \breve p (\boldsymbol n \cdot \bar
{\boldsymbol v}) -s  \bar q (\boldsymbol n \cdot \breve {\boldsymbol u}) \right)
{\rm d} \gamma. \label{slf}
\end{align}

\begin{theorem}\label{wps}
 The variational problem \eqref{ivp} has a unique weak solution $(\breve p,
\breve {\boldsymbol u}) \in H^1 (\Omega_1) \times H^1_{\Gamma_g}(\Omega_2)^3$,
which satisfies
 \begin{align}
\label{sep} \|\nabla \breve {p}\|_{L^2 (\Omega_1)^3} +\|s \breve p\|_{L^2
(\Omega_1)}
&\lesssim \|\breve {\boldsymbol j}\|_{L^2(\Omega_2)^3},\\
\label{seu}  \|\nabla \breve {{\boldsymbol u}}\|_{F(\Omega_2)}
 +\|\nabla \cdot \breve {\boldsymbol u}\|_{L^2(\Omega_2)}+ \| s\breve
{\boldsymbol u}\|_{L^2 (\Omega_2)^3} &\lesssim \frac{1}{ |s|} \|\breve
{\boldsymbol j}\|_{L^2(\Omega_2)^3}.
 \end{align}

\end{theorem}

\begin{proof}
We have from the Cauchy--Schwarz inequality and Lemmas
\ref{tr}--\ref{cc} that
\begin{align*}
 |a\left( \breve p, \breve {\boldsymbol u}; q, \boldsymbol v \right)|
 \leq&\frac{1}{|s|} \|\nabla \breve p\|_{L^2 (\Omega_1)^3} \|\nabla
{q}\|_{L^2(\Omega_1)^3} + \frac{|s|}{c^2} \|\breve p\|_{L^2 (\Omega_1)} \|  q
\|_{L^2 (\Omega_1)}\\
 &\quad +  \rho_1 |s| \left( \mu \| \nabla \breve {\boldsymbol u}\|_{F
(\Omega_2)} \|\nabla {\boldsymbol v}\|_{F (\Omega_2)}
 + (\lambda+ \mu) \| \nabla \cdot \breve {\boldsymbol u}\|_{L^2 (\Omega)} \|
\nabla \cdot {\boldsymbol v}\|_{L^2 (\Omega)}\right) \\
 &\quad + \rho_1 \rho_2 |s|^3 \| \breve {\boldsymbol u}\|_{L^2 (\Omega_2)^3}
\|\boldsymbol v\|_{L^2 (\Omega_2)^3}
 + \frac{1}{|s|} \|\mathscr{B} \breve p\|_{H^{-1/2} (\Gamma_h)} \|q\|_{H^{1/2}
(\Gamma_h)}\\
 &\quad + \rho_1 |s| \left(\|p\|_{L^2 (\Gamma_f) }
\|\boldsymbol n\cdot\boldsymbol v\|_{L^2 (\Gamma_f)}
 +\| q\|_{L^2 (\Gamma_f)} \| \boldsymbol n \cdot \breve {\boldsymbol
u}\|_{L^2 (\Gamma_f)}  \right)\\
 \lesssim & \| \breve p\|_{H^1 (\Omega_1)} \|q\|_{H^1 (\Omega_1)} +\| \breve
{\boldsymbol u}\|_{H^1 (\Omega_2)^3}
 \| \boldsymbol v\|_{H^1 (\Omega_2)^3} +\|\breve
p\|_{H^{1/2}(\Gamma_h)}\|q\|_{H^{1/2}(\Gamma_h)}\\
 &\quad + \|\breve p \|_{H^{1/2} (\Gamma_f)}
 \| \boldsymbol v\|_{H^{1/2} (\Gamma_f)^3}+\|q\|_{H^{1/2}(\Gamma_f)}\|\breve
{\boldsymbol u}\|_{H^{1/2}(\Gamma_f)^3}\\
 \lesssim & \| \breve p\|_{H^1 (\Omega_1)} \|q\|_{H^1 (\Omega_1)} +\| \breve
{\boldsymbol u}\|_{H^1 (\Omega_2)^3}
 \| \boldsymbol v\|_{H^1 (\Omega_2)^3} +\|\breve
p\|_{H^1(\Omega_1)}\|q\|_{H^1(\Omega_1)}\\
 &\quad +\|\breve p\|_{H^1(\Omega_1)}
 \|\boldsymbol v\|_{H^1(\Omega_2)^3} +\|q\|_{H^1 (\Omega_1)}
\|\breve {\boldsymbol u}\|_{H^1(\Omega_2)^3},
\end{align*}
which shows that the sesquilinear form is bounded.

Letting $(q, \boldsymbol v) =(\breve p, \breve {\boldsymbol u})$ in \eqref{slf}
yields
\begin{align}
 a(\breve p, \breve {\boldsymbol u};~ \breve p, \breve {\boldsymbol
u})=& \int_{\Omega_1} \left( \frac{1}{s} |\nabla \breve p|^2 + \frac{s}{c^2}
|\breve p|^2\right){\rm d} \boldsymbol  x +
 \int_{\Omega_2} \big(  \rho_1 \bar s \left( \mu (\nabla \breve {\boldsymbol u}:
\nabla \bar {\breve {\boldsymbol u}})+
 (\lambda+\mu) |\nabla \cdot \breve {\boldsymbol u}|^2\right) \nonumber\\
 &+\rho_1 \rho_2 s |s|^2 |\breve {\boldsymbol u}|^2\big) {\rm d}
\boldsymbol x -\langle s^{-1}\mathscr{B} \breve p, \breve p \rangle_{\Gamma_h}
 +\rho_1 \int_{\Gamma_f} \left( \bar s \breve p (\boldsymbol n \cdot \bar{
\breve{\boldsymbol u}}) - s \bar {\breve p}( \boldsymbol n \cdot \breve
{\boldsymbol u})\right) {\rm d}
\gamma.        \label{pqu}
\end{align}
Taking the real part of \eqref{pqu} and using Lemma \ref{tp}, we obtain
\begin{align}
 {\rm Re } ( a(\breve p, \breve {\boldsymbol u}; \breve p, \breve {\boldsymbol
u}) ) =&\int_{\Omega_1}\left( \frac{s_1}{|s|^2}|\nabla \breve p|^2
+\frac{s_1}{c^2 } |\breve p|^2\right) {\rm d} \boldsymbol  x
 +\rho_1 s_1 \Big( \|\nabla \breve {{\boldsymbol u}}\|^2_{F(\Omega_2)}\notag\\
&\quad +(\lambda+\mu) \|\nabla \cdot \breve {\boldsymbol u}\|^2_{L^2(\Omega_2)}
\Big)+\rho_1 \rho_2  s_1 |s|^2  |\breve {\boldsymbol u}|^2_{L^2
(\Omega_2)^3} - {\rm Re } \langle s^{-1} \mathscr{B} \breve p, \breve p
\rangle_{\Gamma_h} \notag\\
 \gtrsim &  \frac{s_1}{|s|^2 }
  \left( \|\nabla \breve {p}\|^2_{L^2 (\Omega_1)^3} +\|s \breve p\|^2_{L^2
(\Omega_1)} \right) \notag\\
&\quad + s_1   \left( \|\nabla \breve {{\boldsymbol u}}\|^2_{F(\Omega_2)}
+\|\nabla \cdot \breve {\boldsymbol u}\|^2_{L^2(\Omega_2)}+ \| s\breve
{\boldsymbol u}\|^2_{L^2 (\Omega_2)^3} \right).\label{RP}
\end{align}
It follows from the Lax--Milgram lemma that the variational problem \eqref{ivp}
has a unique weak solution $(\breve p, \breve {\boldsymbol u}) \in H^1(\Omega_1)
\times H^1_{\Gamma_g} (\Omega_2)^3.$

Moreover, we have from \eqref{ivp}  that
\begin{align}\label{RP2}
 |a(\breve p, \breve {\boldsymbol u}; \breve p, \breve {\boldsymbol u}) |
 \lesssim \frac{s_1}{|s|} \|\breve {\boldsymbol j}\|_{L^2 (\Omega_2)^3} \| s
\breve {\boldsymbol u}\|_{L^2 (\Omega_2)^3 }.
\end{align}
Combing \eqref{RP} and \eqref{RP2}  leads to
\begin{align*}
\|\nabla \breve {{\boldsymbol u}}\|^2_{F(\Omega_2)}
& +\|\nabla \cdot \breve {\boldsymbol u}\|^2_{L^2(\Omega_2)}+ \| s\breve
{\boldsymbol u}\|^2_{L^2 (\Omega_2)^3} \\
&\lesssim  \frac{1}{s_1} |a(\breve p, \breve {\boldsymbol u}; \breve p, \breve
{\boldsymbol u}) | \lesssim \frac{1}{|s|} \|\breve {\boldsymbol j}\|_{L^2
(\Omega_2)^3} \| s \breve {\boldsymbol u}\|_{L^2 (\Omega_2)^3 }
\end{align*}
and
\begin{align*}
 \frac{1}{|s|^2} \Big( \|\nabla \breve {p}\|^2_{L^2 (\Omega_1)^3} &+\|s \breve
p\|^2_{L^2 (\Omega_1)} \Big) +|s \breve {\boldsymbol u}|^2_{L^2
(\Omega_2)^3}\\
&\lesssim \frac{1}{s_1} |a(\breve p, \breve {\boldsymbol u}; \breve p,
\breve {\boldsymbol u}) | \lesssim \frac{1}{|s|} \|\breve {\boldsymbol j}\|_{L^2
(\Omega_2)^3} \| s \breve {\boldsymbol u}\|_{L^2 (\Omega_2)^3 }.
\end{align*}
Using the Cauchy--Schwarz inequality, we obtain
\begin{align*}
 \|\nabla \breve {{\boldsymbol u}}\|_{F(\Omega_2)}
 +\|\nabla \cdot \breve {\boldsymbol u}\|_{L^2(\Omega_2)}+ \| s\breve
{\boldsymbol u}\|_{L^2 (\Omega_2)^3} \lesssim \frac{1}{ |s|} \|\breve
{\boldsymbol j}\|_{L^2(\Omega_2)^3}
\end{align*}
and
\begin{align*}
 \frac{1}{|s|} \left( \|\nabla \breve {p}\|_{L^2 (\Omega_1)^3} +\|s \breve
p\|_{L^2 (\Omega_1)} \right)& \lesssim  \frac{1}{|s|} \left( \|\nabla \breve
{p}\|_{L^2 (\Omega_1)^3} +\|s \breve p\|_{L^2 (\Omega_1)} \right) +\| s \breve
{\boldsymbol u}\|_{L^2 (\Omega_2)^3 }\\
&\lesssim \frac{1}{ |s|} \|\breve {\boldsymbol j}\|_{L^2(\Omega_2)^3},
\end{align*}
which completes the proof.
\end{proof}

\subsection{Well-posedness in the time-domain}

We now consider the reduced problem in the time-domain:
\begin{subequations}\label{trp}
\begin{numcases}{}
\label{trp1} \Delta p  - \frac{1}{c^2 } \partial_t^2  p =0 &\quad  \text{in}~
$\Omega_1,\,t>0$ \\
\label{trp2} \mu \Delta \boldsymbol u +  (\lambda+\mu) \nabla \nabla \cdot
\boldsymbol u -\rho_2 \partial_t^2 \boldsymbol u =\boldsymbol j
 & \quad  \text {in}~  $\Omega_2 ,\,t>0$, \\
\label{trp3} p|_{t=0}=\partial_t p|_{t=0}=0,\quad \boldsymbol
u|_{t=0}=\partial_{t} \boldsymbol u|_{t=0}=0 &\quad  \text {in}~ $\Omega$,\\
 \label{trp4}\partial_{\boldsymbol n} p=-\rho_1 \boldsymbol n
\cdot \partial_t^2 \boldsymbol u,  \quad -p \boldsymbol n = \boldsymbol \sigma
(\boldsymbol u) \cdot \boldsymbol n &\quad \text{on}~ $\Gamma_f,\,t>0$,\\
\label{trp5} \partial_{\boldsymbol \nu} p= \mathscr T p &\quad
\text{on}~$\Gamma_h,\,t>0$,\\
 \label{trp6} \boldsymbol u =0 &\quad  \text {on}~ $\Gamma_g,\,t>0$.
\end{numcases}
\end{subequations}
To show the well-posedness of the reduced problem \eqref{trp}, we make the
following assumption for the source term $ \boldsymbol j:$
\begin{align}\label{ash}
 \boldsymbol j \in H^1 (0, T; L^2 (\Omega_2)^3), \quad \boldsymbol j
\big|_{t=0}=0.
\end{align}

\begin{theorem}
The initial-boundary value problem \eqref{trp} has a unique solution
$\left(p, \boldsymbol u\right)$ which satisfies
\begin{align*}
 & p (\boldsymbol x, t ) \in L^2 \left(0, T;~ H^1(\Omega_1) \right)\cap H^1
\left (0, T;~ L^2 (\Omega_1) \right),\\
 & \boldsymbol u (\boldsymbol x, t) \in L^2  \big(0, T;~ H^1_{\Gamma_g}
(\Omega_2)^3\big) \cap H^1 \big( 0, T;~ L^2 (\Omega_2)^3 \big)
\end{align*}
and the stability estimates
\begin{align}
&\max \limits_{t\in [0, T]}
 \left( \| \partial_t p\|_{L^2 (\Omega_1)} + \|\nabla p\|_{L^2 (\Omega_1)^3 }
\right)  \lesssim  \|\partial_t \boldsymbol j \|_{ L^1 (0, T;~ L^2
(\Omega_2)^3)}, \label{es1} \\
 &\max \limits_{ t \in [0, T]}
 \left(  \| \partial_t \boldsymbol u\|_{L^2 (\Omega_2)^3} +\| \nabla \cdot
\boldsymbol u\|_{L^2 (\Omega_2)}
 +\| \nabla  \boldsymbol u\|_{F (\Omega_2)} \right)  \lesssim  \|\partial_t
\boldsymbol j \|_{ L^1 (0, T; ~L^2 (\Omega_2)^3)}. \label{es2}
\end{align}
\end{theorem}

\begin{proof}
For the air/fluid pressure $p$,  we have
\begin{align*}
\int_0^T &\left( \|\nabla p\|_{L^2(\Omega_1)^3}^2+\|\partial_t
p\|^2_{L^2(\Omega_1)}\right) {\rm d} t\\
&\leq  \int_0^T e^{- 2 s_1(t- T)}\left (  \|\nabla p\|_{L^2(\Omega_1)^3}^2+\|\partial_t p\|^2_{L^2(
\Omega_1)}\right) {\rm d} t\\
&= e^{2 s_1 T} \int_0^T e^{-2 s_1 t} \left ( \|\nabla p\|_{L^2(\Omega_1)^3}^2
+\|\partial_t p\|^2_{L^2(\Omega_1)}\right) {\rm d} t\\
&\lesssim  \int_0^{\infty} e^{-2 s_1 t}  \left ( \|\nabla
p\|_{L^2(\Omega_1)^2}^3+\|\partial_t p\|^2_{L^2(\Omega_1)}\right) {\rm d} t.
\end{align*}
Similarly, we have for the elastic displacement $\boldsymbol u$ that
\begin{align*}
\int_0^T &\left(\|\partial_t \boldsymbol u\|^2_{L^2 (\Omega_2)^3}
+\| \nabla \boldsymbol u\|^2_{F (\Omega_2)} +\| \nabla \cdot \boldsymbol
u\|^2_{L^2 (\Omega_2)} \right) {\rm d} t\\
&\leq  \int_0^T e^{-2 s_1 (t-T)} \left( \|\partial_t \boldsymbol u\|^2_{L^2
(\Omega_2)^3} +\| \nabla \boldsymbol u\|^2_{F (\Omega_2)} +\| \nabla \cdot
\boldsymbol u\|^2_{L^2 (\Omega_2)} \right) {\rm d} t\\
&= e^{2s_1 T} \int_0^T e^{-2 s_1 t} \left( \|\partial_t \boldsymbol u\|^2_{L^2
(\Omega_2)^3} +\| \nabla \boldsymbol e\|^2_{F (\Omega_2)} +\| \nabla \cdot
\boldsymbol u\|^2_{L^2 (\Omega_2)} \right) {\rm d} t\\
&\lesssim\int_0^\infty e^{-2 s_1 t} \left( \|\partial_t \boldsymbol u\|^2_{L^2
(\Omega_2)^3} +\| \nabla \boldsymbol u\|^2_{F (\Omega_2)} +\| \nabla \cdot
\boldsymbol u\|^2_{L^2 (\Omega_2)} \right) {\rm d} t.
\end{align*}
Hence it suffices to estimate the integrals
\[
\int_0^{\infty} e^{-2 s_1 t}  \big ( \|\nabla p\|_{L^2(\Omega_1)^3}^2
+\|\partial_t p\|^2_{L^2(\Omega_1)}\big) {\rm d} t
\]
and
\[
\int_0^\infty e^{-2 s_1 t} \left( \|\partial_t \boldsymbol u\|^2_{L^2
(\Omega_2)^3} +\| \nabla \boldsymbol u\|^2_{F (\Omega_2)} +\| \nabla \cdot
\boldsymbol u\|^2_{L^2 (\Omega_2)} \right) {\rm d} t.
\]

Taking the Laplace transform of \eqref{trp}, we obtain the reduced
acoustic-elastic interaction problem in the $s$-domain \eqref{saep}. It follows
from Theorem \ref{wps} that $\breve {p}$ and $\breve {\boldsymbol u}$ satisfy
the stability estimates \eqref{sep} and \eqref{seu}, respectively. It follows
from \cite[Lemma 44.1]{Treves1975} that $\breve p$ and $\breve{\boldsymbol u}$
are holomorphic functions of $s$ on the half plane $s_1 >\bar\gamma>0,$  where
$\bar \gamma$ is any positive constant. Hence we have from Lemma \ref{A2} that
the inverse Laplace transform of $\breve p$ and $\breve{\boldsymbol u}$ exist
and are supported in $[0, \infty].$

Using the Parseval identity \eqref{PI}, the assumptions \eqref{ash}, and the
stability estimate \eqref{sep}, we have
\begin{align*}
 \int_0^{\infty} e^{-2s_1 t}
 &\left(\|\nabla p\|^2_{L^2 (\Omega_1)^3}+\|\partial_t p
\|^2_{L^2(\Omega_1)}\right){\rm d}t
 =\frac{1}{2 \pi} \int_{-\infty}^{\infty}\left( \|\nabla \breve
p\|^2_{L^2(\Omega_1)^3}+ \|s \breve p\|^2_{L^2(\Omega_1)}\right) {\rm d} s_2\\
 \lesssim&  s_1 ^{-2} \int_{-\infty}^{\infty}  \| |s|\breve {\boldsymbol
j}\|^2_{L^2 (\Omega_2)^3}  {\rm d}s_2
  =s_1^{-2} \int_{-\infty}^{\infty} \|\mathscr L  ( \partial_t \boldsymbol
j)\|^2_{L^2 (\Omega_2)^3} {\rm d} s_2\\
\lesssim &s_1^{-2} \int_0^{\infty} e^{-2 s_1 t} \| \partial_t \boldsymbol
j\|^2_{L^2 (\Omega_2)^3} {\rm d} t,
 \end{align*}
which shows that
\begin{align*}
 p (\boldsymbol x, t) \in  L^2 \left(0, T; H^1(\Omega_1) \right)\cap H^1 \big(0,
T; L^2 (\Omega_1)\big).
\end{align*}
Since $\breve {\boldsymbol u} =\mathscr L (\boldsymbol u)= \mathscr F (e^{-s_1
t} \boldsymbol u)$, where $\mathscr F$ is the Fourier transform in $s_2$, we
have from the Parseval identity \eqref{PI} and the stability estimate
\eqref{seu} that
\begin{align*}
 \int_{0}^{\infty} & e^{-2 s_1 t} \left( \|\partial_t \boldsymbol u\|^2_{L^2
(\Omega_2)^3} +\| \nabla \boldsymbol u\|^2_{F (\Omega_2)} +\| \nabla \cdot
\boldsymbol u\|^2_{L^2 (\Omega_2)} \right) {\rm d} t \\
&=\frac{1}{2\pi} \int_{-\infty}^{\infty}\left( \|s \breve {\boldsymbol
u}\|^2_{L^2 (\Omega_2)^3}
+\|\nabla \breve {\boldsymbol u}\|^2_{F (\Omega_2)}+\|\nabla \cdot \breve
{\boldsymbol u}\|^2_{L^2 (\Omega_2)} \right) {\rm d} s_2\\
&\lesssim s_1^{-2} \int_{-\infty}^{\infty} \|\breve {\boldsymbol j}\|^2_{L^2
(\Omega_2)^3} {\rm d} s_2
= s_1^{-2} \int_{-\infty}^{\infty} \| \mathscr L (\boldsymbol j)\|^2_{L^2
(\Omega_2)^3} {\rm d} s_2\\
&\lesssim s_1^{-2} \int_0^{\infty} e^{-2 s_1 t} \|\boldsymbol j\|^2_{L^2
(\Omega_2)^3} {\rm d} t.
\end{align*}
It follows from \eqref{f1} that
\[
 \boldsymbol u  (\boldsymbol x, t)\in L^2  \big(0, T; H^1_{\Gamma_g}
(\Omega_2)^3\big) \cap H^1 \big( 0, T; L^2 (\Omega_2)^3 \big).
\]

Next we show the stability estimates. Let $\tilde p$ be the extension of $p$
with respect to $t$ in $\mathbb R$ such that $\breve p=0$ outside the interval
$[0, t].$  By the Parseval identity \eqref{PI} and Lemma \ref{tp}, we get
\begin{align*}
& {\rm Re} \int_{0}^{t} e^{-2 s_1 t} \langle \mathscr T  p,  {\partial_t \bar p}
\rangle_{\Gamma_h}{\rm d}t = {\rm Re} \int_0^t e^{-2 s_1 t} \int_{\Gamma_h}
(\mathscr T p)  {\partial_t \bar p} {\rm d} \boldsymbol r {\rm d}t\\
=& {\rm Re} \int_{\Gamma_h} \int_0^{\infty} e^{-2 s_1 t} (\mathscr T \tilde p)
{\partial_t \bar { \tilde p}} {\rm d} t{\rm d}  \boldsymbol r =\frac{1}{2 \pi}
\int_{-\infty}^{\infty} {\rm Re} \langle \mathscr{B} \breve {\tilde p},
~s \breve {\tilde p} \rangle_{\Gamma_h} {\rm d} s_2\\
 =&\frac{1}{2\pi} \int_{-\infty}^{\infty} |s|^2 {\rm Re} \langle s^{-1}
\mathscr{B} \breve {\tilde p}, ~\breve {\tilde p} \rangle_{\Gamma_h} {\rm d}
s_2 \leq 0,
\end{align*}
which yields after taking $s_1 \rightarrow 0$ that
\begin{align}\label{tp1}
 {\rm Re} \int_0^t  \int_{\Gamma_h} (\mathscr T p){\partial_t \bar p}{\rm
d}\boldsymbol r {d}t \leq 0.
\end{align}

Taking the partial derivative of \eqref{trp2}--\eqref{trp4} and \eqref{trp6}
with respect to $t$, we get
\begin{align}\label{fdt}
 \begin{cases}
  \mu \Delta  (\partial_t \boldsymbol u) + (\lambda +\mu ) \nabla \nabla \cdot
(\partial_t \boldsymbol u) -  \rho_2 \partial_t^2 (\partial_t \boldsymbol u) =
\partial_t\boldsymbol j \quad &  \text {in} ~\Omega_2,\,t >0,\\
  \partial_t \boldsymbol u|_{t=0}=0  \quad &\text{in} ~ \Omega_2,\\
  \partial_{t}^2 \boldsymbol u|_{t=0}= \rho_{2}^{-1 } \left(
  \mu \Delta \boldsymbol u + (\lambda+\mu) \nabla \nabla \cdot \boldsymbol u
-\boldsymbol j  \right)|_{t=0}=0 \quad & \text {in}~ \Omega_2,\\
  -\partial_t p\,\boldsymbol n = \partial_t (\boldsymbol \sigma (\boldsymbol u))
\cdot \boldsymbol n  =  \boldsymbol \sigma (\partial_t \boldsymbol u) \cdot
\boldsymbol n \quad &\text {on}~\Gamma_f,\,t>0,\\
  \partial_t \boldsymbol u =0 \quad & \text {on }~\Gamma_g,\,t>0.
 \end{cases}
\end{align}
For any  $0< t < T$, consider the energy  function
\begin{align*}
 \mathscr E (t)= e_1 (t) +e_2 (t),
\end{align*}
where
\[
  e_1 (t)= \| \frac{1}{c} \partial_t p\|^2_{ L^2 (\Omega_1)} +\| \nabla
p\|^2_{L^2 (\Omega_1)^3}
\]
and
\begin{align*}
 e_2 (t)= \|(\rho_1 \rho_2)^{1/2} \partial_t ^2  \boldsymbol u\|^2_{L^2
(\Omega_2)^3} &+\| (\rho_1 (\lambda + \mu))^{1/2} \nabla \cdot (\partial_t
\boldsymbol u)\|^2_{L^2 (\Omega_2)}\\
&+\|(\rho_1\mu)^{1/2}\nabla (\partial_t \boldsymbol u)\|^2_{F (\Omega_2)}.
\end{align*}
It is easy to note that
\begin{align} \label{e00}
 \mathscr E (t) -\mathscr E (0)
 =\int_0^t \mathscr E ' (\tau) {\rm d} \tau = \int_0^t \left ( e_1' (\tau)
+e_2'(\tau) \right) {\rm d} \tau.
 \end{align}
It follows from \eqref{trp1}, \eqref{trp3}--\eqref{trp5} and the integration
by parts that
\begin{align}
\int_0^t e_1' (\tau) {\rm d} \tau
=& 2 {\rm Re } \int_0^t \int_{\Omega_1} \left (\frac{1}{c^2} \partial_t^2 p  ~
\partial_t \bar p + \partial_t  (\nabla p )\cdot  \nabla \bar p \right) {\rm d}
\boldsymbol  x  {\rm d}\tau \nonumber\\
=& 2 {\rm Re} \int_0^t \int_{\Omega_1} \left ( \Delta p  \partial_t \bar p +
\partial_t  (\nabla p )\cdot  \nabla \bar p \right) {\rm d} \boldsymbol x  {\rm
d}\tau \nonumber\\
=& \int_0^{t} \int_{\Omega_1 } 2 {\rm Re} \left (- \nabla p \cdot \partial_t
(\nabla \bar {p}) + \partial_t (\nabla p) \cdot \nabla \bar p \right) {\rm d}
\boldsymbol  x {\rm d} \tau \nonumber \\
&\quad+ 2 {\rm Re}\int_0^t \int_{\Gamma_h} (\mathscr T p)  \partial_t \bar p
{\rm d} \boldsymbol r {\rm d} \tau
- 2 {\rm Re} \int_0^t \int_{\Gamma_f} \partial_{\boldsymbol n} p
\partial_t \bar p  {\rm d}\gamma {\rm d} \tau\nonumber\\
=&   2 {\rm Re}\int_0^t \int_{\Gamma_h} (\mathscr T p)  \partial_t \bar p {\rm
d} \boldsymbol r{\rm d} \tau + 2 {\rm Re} \int_0^t \int_{\Gamma_f}
\rho_1 \boldsymbol n\cdot \partial_t^2 \boldsymbol u \partial_t  \bar p
 {\rm d} \gamma {\rm d}\tau. \label{e11}
\end{align}
Similarly, we have from \eqref{fdt} and the integration by parts that
\begin{align}
 \int_0^{t} e'_{2} (\tau ) {\rm d} \tau
 =&  \rho_1 2 {\rm Re} \int_0^t \int_{\Omega_2} \big(
  \rho_2 \partial_t (\partial_t^2 \boldsymbol u ) \cdot  \partial_t^2 \bar
{\boldsymbol u}
 + (\lambda +\mu) \nabla \cdot (\partial_t^2 \boldsymbol u) \nabla \cdot
(\partial_t  \bar {\boldsymbol u})\notag\\
&+ \mu  \nabla (\partial_t^2 \boldsymbol u) :\nabla (\partial_t \bar
{\boldsymbol u}) \big){\rm d} \boldsymbol x {\rm d} \tau \nonumber\\
 =& {\rho_1} 2 {\rm Re} \int_0^ t\int_{\Omega_2} \big(
 \left( \mu \Delta  (\partial_t \boldsymbol u) + (\lambda +\mu ) \nabla \nabla
\cdot (\partial_t \boldsymbol u)
 -\partial_t \boldsymbol j \right)  \cdot  \partial_t^2 \bar {\boldsymbol u}
\nonumber\\
 &\quad+(\lambda +\mu) \nabla \cdot (\partial_t^2 \boldsymbol u) \nabla \cdot
(\partial_t  \bar {\boldsymbol u}) + \mu  \nabla (\partial_t^2 \boldsymbol u) :
\nabla (\partial_t \bar {\boldsymbol u})
 \big){\rm d}\boldsymbol x {\rm d} \tau \nonumber \\
 =& \rho_1 \int_0^t \int_{\Omega_2 }{\rm Re} \big( - \mu  \nabla (\partial_t
\boldsymbol u) : \nabla (\partial_t^2  \bar {\boldsymbol u}) -
 (\lambda+\mu) \nabla \cdot (\partial_t \boldsymbol u) \nabla \cdot
(\partial_t^2  \bar {\boldsymbol u})\nonumber\\
 &\quad+ (\lambda +\mu) \nabla \cdot (\partial_t^2 \boldsymbol u) \nabla \cdot
(\partial_t  \bar {\boldsymbol u})
 + \mu  \nabla (\partial_t^2 \boldsymbol u) : \nabla (\partial_t \bar
{\boldsymbol u})\big) {\rm d} \boldsymbol x  {\rm d} \tau \nonumber\\
 &\quad- 2 {\rm Re} \rho_1\int_0^t \int_{\Omega_2} \partial_t \boldsymbol j
\cdot \partial_t^2 \bar {\boldsymbol u} {\rm d} \boldsymbol x  {\rm d} \tau
 + 2 {\rm Re} {\rho_1 }\int_0^t \int_{\Gamma_f} (\boldsymbol \sigma (\partial_t
\boldsymbol u) \cdot \boldsymbol n) \cdot \partial_t^2 \bar {\boldsymbol u} {\rm
d} \gamma {\rm d} \tau  \nonumber  \\
 =& - 2 {\rm Re} \rho_1\int_0^t \int_{\Omega_2} \partial_t \boldsymbol j  \cdot
\partial_t^2 \bar {\boldsymbol u} {\rm d} \boldsymbol x  {\rm d} \tau -
  2 {\rm Re }\rho_1 \int_0^t \int_{\Gamma_f} \partial_t p \boldsymbol n \cdot
\partial_t^2  \bar {\boldsymbol u} {\rm d} \gamma {\rm d} \tau. \label{e22}
\end{align}
Since $\mathscr E (0)=0,$ combining \eqref{e00}-- \eqref{e22} and \eqref{tp1} gives
\begin{align*}
\mathscr E (t)
&= 2 {\rm Re} \int_0^t \int_{\Gamma_h} (\mathscr T p) \partial_t  \bar p {\rm d}
\boldsymbol r  {\rm d} \tau  -2 {\rm Re} \rho_1 \int_0^t \int_{\Omega_2}
\partial_t \boldsymbol j \cdot
\partial_t^2 \bar {\boldsymbol u} {\rm d} \boldsymbol  x  {\rm d} \tau\\
& \leq   - 2 {\rm Re} \rho_1\int_0^t \int_{\Omega_2} \partial_t \boldsymbol j
\cdot\partial_t^2 \bar {\boldsymbol u} {\rm d} \boldsymbol  x  {\rm d} \tau \\
& \leq 2 \rho_1 \max \limits_{t \in [0, T]} \|\partial_t^2  {\boldsymbol
u}\|_{L^2 (\Omega_2)^3} \|\partial_t \boldsymbol j \|_{ L^1 (0, T; L^2
(\Omega_2)^3)}.
\end{align*}
Thus, we can obtain the estimate for the air/fluid pressure $p$:
\begin{align*}
 \max \limits_{t\in[0, T]}&
 \left( \| \partial_t p\|^2_{L^2 (\Omega_1)} + \|\nabla p\|^2_{L^2 (\Omega_1)^3
}  \right)\\
&\leq \max \limits_{t\in[0, T]} \left( \| \partial_t p\|^2_{L^2 (\Omega_1)} +
\|\nabla p\|^2_{L^2 (\Omega_1)^3 } + \|\partial_t^2 \boldsymbol u \|^2_{L^2
(\Omega_2)^3} \right)\\
&\lesssim \max \limits_{t\in[0, T]} \mathscr E (t)
 \lesssim  \max \limits_{t \in [0, T]} \|\partial_t^2  {\boldsymbol u}\|_{L^2
(\Omega_2)^3} \|\partial_t \boldsymbol j \|_{L^1 (0, T; L^2 (\Omega_2)^3)}.
\end{align*}
It follows from Young's inequality that
\begin{align*}
 \max \limits_{t\in[0, T]}
 \left( \| \partial_t p\|_{L^2 (\Omega_1)} + \|\nabla p\|_{L^2 (\Omega_1)^3 }
\right)  \lesssim  \|\partial_t \boldsymbol j \|_{ L^1 (0, T; L^2(\Omega_2)^3)},
\end{align*}
which shows the stability estimate \eqref{es1}.

For the elastic displacement $\boldsymbol u$, we can also obtain
\begin{align*}
 \max \limits_{ t \in [0, T]}
 &\left(  \| \partial_t^2 \boldsymbol u\|^2_{L^2 (\Omega_2)^3} +\| \nabla \cdot
(\partial_t \boldsymbol u)\|^2_{L^2 (\Omega_2)}
 +\| \nabla (\partial_t \boldsymbol u)\|^2_{F (\Omega_2)} \right)\\
& \lesssim \max
\limits_{[0, T]} \mathscr E (t)\lesssim  \max \limits_{t \in [0, T]}
\|\partial_t^2  {\boldsymbol u}\|_{L^2
(\Omega_2)^3} \|\partial_t \boldsymbol j \|_{ L^1 (0, T; L^2 (\Omega_2)^3)}.
\end{align*}
It follows from the Cauchy--Schwarz inequality that
\begin{align}\label{EP1}
 \max \limits_{ t \in [0, T]}
 \left(  \| \partial_t^2 \boldsymbol u\|^2_{L^2 (\Omega_2)^3} +\| \nabla \cdot
(\partial_t \boldsymbol u)\|^2_{L^2 (\Omega_2)}
 +\| \nabla (\partial_t \boldsymbol u)\|^2_{F (\Omega_2)} \right)\notag\\
 \lesssim\|\partial_t \boldsymbol j\|^2_{L^1 (0, T; L^2 (\Omega_2)^3)}.
\end{align}
For any $0 <t \leq T$, using the epsilon inequality leads to
\begin{align*}
 \|\partial_t \boldsymbol u\|^2_{L^2 (\Omega_2)^3} = \int_0^t \partial_\tau
\|\partial_\tau \boldsymbol u (\cdot, \tau)\|^2_{L^2 (\Omega_2)^3}  {\rm d} \tau
  \leq \epsilon T  \|\partial_t \boldsymbol u\|^2_{L^2 (\Omega_2)^3} +
\frac{T}{\epsilon} \| \partial_t^2 \boldsymbol u\|^2_{L^2 (\Omega_2)^3}.
\end{align*}
Here we choose $\epsilon>0 $ small enough  such that  $ \epsilon T <1$,
e.g., $\epsilon =\frac{1}{2T}$. Hence we have
\begin{align}\label{ep1}
 \|\partial_t \boldsymbol u\|^2_{L^2 (\Omega_2)^3} \lesssim \| \partial_t^2
\boldsymbol u\|^2_{L^2 (\Omega_2)^3}.
\end{align}
Similarly, we can  obtain
\begin{align}\label{ep2}
 \|\nabla \cdot \boldsymbol u\|^2_{L^2 (\Omega_2)} \lesssim  \| \nabla \cdot
(\partial_t \boldsymbol u)\|^2_{L^2 (\Omega_2)}, \quad
 \|\nabla \boldsymbol u\|^2_{F(\Omega_2)} \lesssim \| \nabla (\partial_t
\boldsymbol u)\|^2_{F (\Omega_2)}.
\end{align}
Combining \eqref{EP1}--\eqref{ep2} gives
\begin{align*}
 \max \limits_{t\in [0, T]} \left( \|\partial_t \boldsymbol u\|^2_{L^2
(\Omega_2)^3} +\|\nabla \cdot \boldsymbol u\|^2_{L^2 (\Omega_2)} +  \|\nabla
\boldsymbol u\|^2_{F(\Omega_2)}\right) 
\lesssim\|\partial_t\boldsymbol j\|^2_{L^1(0,T;L^2(\Omega_2)^3)},
\end{align*}
which shows the estimate \eqref{es2}.
\end{proof}

\subsection{A priori estimates}

In what follows, we derive a priori stability estimates for the  air/fluid
pressure $p$ and the displacement $\boldsymbol u$ with a minimum regularity
requirement for the data and an explicit dependence on the time.

We shall consider the elastic wave equation for  $\partial_t \boldsymbol u$ in
order to match the interface conditions when deducing the stability estimates.
Taking the partial derivative  of \eqref{trp2}--\eqref{trp5} and \eqref{trp6}
with respect to $t$, we obtain a new reduced problem:
\begin{align}\label{TRP}
 \begin{cases}
 \Delta p  - \frac{1}{c^2 } \partial_t^2  p =0 \quad
&\text{in}~\Omega_1,\,t>0\\
  \partial_{\boldsymbol \nu} p= \mathscr T p \quad  &\text
{on}~\Gamma_h,\,t>0,\\
 \partial_{\boldsymbol n} p=-\rho_1 \boldsymbol n
\cdot\partial_t^2 \boldsymbol u \quad & \text {on}~\Gamma_f,\,t>0,\\
  p|_{t=0}=\partial_t p|_{t=0}=0\quad & \text {in}~\Omega_1 \\
  \mu \Delta (\partial_t \boldsymbol u) +  (\lambda+\mu) \nabla \nabla \cdot
(\partial_t\boldsymbol u) -\rho_2  \partial_t^2 (\partial_t \boldsymbol u) =
\partial_t\boldsymbol j \quad  & \text {in}~  \Omega_2,\,t>0,\\
  \partial_t \boldsymbol u|_{t=0}=0\quad & \text {in}~\Omega_2,\\
 \partial_{t}^2 \boldsymbol u|_{t=0}= \rho_2^{-1 } \left(
  \mu \Delta \boldsymbol u + (\lambda+\mu) \nabla \nabla \cdot \boldsymbol u
-\boldsymbol j  \right)|_{t=0}=0 \quad & \text {in}~ \Omega_2,\\
   -\partial_t p \boldsymbol n = \partial_t (\boldsymbol \sigma (\boldsymbol
u)) \cdot \boldsymbol n  = \boldsymbol \sigma (\partial_t \boldsymbol u) \cdot
\boldsymbol n \quad &\text {on}~\Gamma_f,\,t>0,\\
  \partial_t \boldsymbol u =0 \quad & \text {on }~\Gamma_g,\,t>0.
 \end{cases}
 \end{align}
The variational problems of \eqref{TRP} is to find $(p, \boldsymbol u) \in
H^1(\Omega_1) \times   \in H^1_{\Gamma_g} (\Omega_2)^3$ for all $t>0$ such that
\begin{align}
 \int_{\Omega_1} \frac{1}{c^2}
 \partial_t^2 p \bar q {\rm d} \boldsymbol  x
 =&-\int_{\Omega_1} \nabla p \cdot \nabla \bar q {\rm d} \boldsymbol  x
+\int_{\Gamma_h} (\mathscr T p) \bar q {\rm d} \boldsymbol r
 -\int_{\Gamma_f} \partial_{\boldsymbol n} p\, \bar q {\rm d}
\gamma \nonumber\\
 =&-\int_{\Omega_1} \nabla p \cdot \nabla \bar q {\rm d} \boldsymbol  x
 +\int_{\Gamma_h} (\mathscr T p) \bar q {\rm d} \boldsymbol r +\int_{\Gamma_f}
\rho_1 (\boldsymbol n \cdot \partial_t^2 \boldsymbol u) \bar {q} {\rm d}\gamma,
\quad \forall q \in H^1 (\Omega_1) \label{v1}
\end{align}
and
\begin{align}
  \int_{\Omega_2}\rho_2 \partial_t^2  (\partial_t \boldsymbol u) \cdot \bar
{\boldsymbol v} {\rm d} \boldsymbol x
 =& -\int_{\Omega_2} \left(\mu \nabla { (\partial_t \boldsymbol u)} : \nabla
{\bar{ \boldsymbol v}} + (\lambda+\mu) (\nabla \cdot (\partial_t \boldsymbol u))
(\nabla \cdot \bar {\boldsymbol v}) \right) {\rm d} \boldsymbol x \nonumber\\
 &\quad-\int_{\Omega_2} \partial_t \boldsymbol j \cdot \bar {\boldsymbol v} {\rm
d} \boldsymbol x + \int_{\Gamma_f}\big(\boldsymbol \sigma (
\partial_t\boldsymbol
u) \cdot\boldsymbol n\big) \cdot \bar {\boldsymbol v} {\rm d} \gamma \nonumber\\
 =&-\int_{\Omega_2} \left(\mu \nabla { (\partial_t\boldsymbol u)} : \nabla
{\bar{ \boldsymbol v}} + (\lambda+\mu) (\nabla \cdot (\partial_t \boldsymbol u))
(\nabla \cdot \bar {\boldsymbol v})
 + (\partial_t\boldsymbol j) \cdot \bar {\boldsymbol v} \right) {\rm d}
\boldsymbol  x \nonumber\\
 &\quad-\int_{\Gamma_f} (\partial_t p)( \boldsymbol n \cdot \bar {\boldsymbol
v}) {\rm d} \gamma, \quad \forall \boldsymbol v\in H^1_{\Gamma_g}(\Omega_2)^3.
\label{v2}
\end{align}

To show the stability of the solution, we follow the argument in
\cite{Treves1975} but with a careful study of the TBC. The following lemma is
useful for the subsequent analysis.

\begin{lemma}\label{ttp}
 Given $\xi \geq 0$ and $p \in H^1 (\Omega_1),$ we have
 \begin{align*}
  {\rm Re} \int_{\Gamma_h} \int_0^{\xi} \left( \int_0^{t} \mathscr T p (\cdot,
\tau) {\rm d} \tau \right) \bar p (\cdot, t) {\rm d} t {\rm d} \boldsymbol r
\leq 0.
 \end{align*}
\end{lemma}

\begin{proof}
Let $\tilde p$ be the extension of $p$ with respect to  $t$ in $\mathbb R$ such
that $\tilde p =0$ outside the interval $[0, \xi].$ We obtain from the
Parseval identity \eqref{PI} and Lemma \ref{tp} that
\begin{align*}
  {\rm Re} \int_{\Gamma_h} &\int_0^{\xi}
  e^{-2 s_1 t} \left( \int_0^{\tau} \mathscr T p (\cdot, \tau) {\rm d} \tau
\right) \bar  p (\cdot, t) {\rm d} t {\rm d} \boldsymbol r\\
&  = {\rm Re} \int_{\Gamma_h} \int_0^{\infty} e^{- 2 s_1 t} \left( \int_0^{t}
\mathscr T  \tilde p (\cdot, \tau) {\rm d} \tau \right) \bar {\tilde p} (\cdot,
t) {\rm d} t {\rm d} \boldsymbol r\\
  &= {\rm Re} \int_{\Gamma_h} \int_0^{\infty} e^{-2 s_1 t} \left( \int_0^t
\mathscr L^{-1} \circ \mathscr{B} \circ \mathscr L
  \tilde p (\cdot, \tau) {\rm d} \tau \right) \bar {\tilde  p} (\cdot, t) {\rm
d} t{\rm d} \boldsymbol r\\
  &={\rm Re} \int_{\Gamma_h} \int_0^{\infty}  e^{- 2 s_1 t} \left( \mathscr
L^{-1} \circ (s^{-1} \mathscr{B})
  \circ \mathscr L  \tilde p (\cdot, t) ~\bar {\tilde p} (\cdot, t)   \right)
{\rm d} t {\rm d} \boldsymbol r\\
  &=\frac{1}{2 \pi} \int_{-\infty}^{\infty} {\rm Re} \int_{\Gamma_h} s^{-1}
\mathscr{B}\breve {\tilde p} (\cdot, s) \bar {\breve {\tilde p}} (\cdot, s) {\rm
d} \boldsymbol r {\rm d} s_2\\
  &=\frac{1}{2\pi} \int_{-\infty}^{\infty} {\rm Re} \langle s^{-1} \mathscr{B}
\breve{\tilde p}, \breve {\tilde p} \rangle_{\Gamma_h} {\rm d} s_2 \leq 0,
\end{align*}
where we have used the fact that
\[
 \int_0^{t} p (\cdot, \tau){\rm d} \tau = \mathscr L^{-1} \left( s^{-1} \breve
{p} (\cdot, s) \right).
\]
The proof is completed after taking the limit $s_1 \to 0$.
\end{proof}

\begin{theorem}
Let $ (p, \boldsymbol u) \in H^1 (\Omega_1) \times H^1_{\Gamma_g} (\Omega_2)^3$
be the solution of \eqref{v1}--\eqref{v2}.  Given $ \partial_t \boldsymbol j
\in L^1 \left( 0, T; ~L^2 (\Omega_2)^3 \right),$  for any $T>0,$ we have
\begin{align}
  \|p\|_{L^{\infty}\left(0, T;~ L^2(\Omega_1) \right)}
 &\lesssim T  \|\partial_t \boldsymbol j\|_{L^1(0, T; ~L^2 (\Omega_2)^3)},
\label{ess1}\\
 \|\boldsymbol u\|_{L^{\infty} (0, T; L^2 (\Omega_2)^3)}
 &\lesssim  T^2  \|\partial_t \boldsymbol j\|_{L^1(0, T; ~L^2 (\Omega_2)^3)}
\label{ess2}\\
 \|p\|_{L^{2}\left(0, T;~ L^2(\Omega_1) \right)}
 &\lesssim T^{3/2}  \|\partial_t \boldsymbol j\|_{L^1(0, T; ~L^2 (\Omega_2)^3)},
\label{ess}\\
\|\boldsymbol u\|_{L^2 (0, T; L^2 (\Omega_2)^3)}&\lesssim   
T^{5/2}\|\partial_t \boldsymbol j\|_{L^1(0, T; ~L^2(\Omega_2)^3)}. \label{ess3}
\end{align}
\end{theorem}

\begin{proof}
 Let $ 0< \theta < T$ and define  an  auxiliary function
 \begin{align*}
  \psi_1 (\boldsymbol x, t) =\int_t^{\theta} p (\boldsymbol x, \tau) {\rm d}
\tau, \quad \boldsymbol x  \in \Omega_1, ~~~ 0 \leq t \leq \theta.
 \end{align*}
It is clear to note that
\begin{align}\label{F1}
\psi_1(\boldsymbol x, \theta)=0, \quad \partial_t\psi_1(\boldsymbol x, t)=-p(\boldsymbol x, t).
\end{align}
For any $\phi (\boldsymbol x, t) \in L^2 \left( 0, \xi;~ L^2 (\Omega_1)
\right)$, we have
\begin{align}\label{F2}
 \int_0^{\theta} \phi (\boldsymbol x, t) \bar {\psi_1} (\boldsymbol x, t) {\rm
d} t= \int_0^{\theta} \left( \int_0^t \phi (\boldsymbol x, \tau) {\rm d} \tau
\right) \bar p (\boldsymbol x, t) {\rm d} t.
\end{align}
Indeed, we have from the integration by parts and \eqref{F1} that
\begin{align*}
 \int_0^{\theta}
 &\phi (\boldsymbol x, t) \bar \psi_1 (\boldsymbol x, t) {\rm d} t
=\int_0^{\theta} \left( \phi (\boldsymbol x, t) \int_t^{\theta} \bar p
(\boldsymbol x, \tau) {\rm d} \tau \right) {\rm d} t\\
 &=\int_0^{\theta}\int_t^{\theta} \bar p (\boldsymbol x, \tau) {\rm d} \tau {\rm
d} \left( \int_0^{t} \phi (\boldsymbol x, \varsigma) {\rm d}\varsigma\right)\\
 &=\int_{t}^{\theta} \bar p (\boldsymbol x, \tau) {\rm d} \tau \int_0^{t} \phi
(\boldsymbol x, \varsigma) {\rm d} \varsigma \big|_0^{\theta}
 +\int_0^{\theta} \left( \int_0^{t} \phi (\boldsymbol x,\varsigma) {\rm d}
\varsigma \right) \bar p (\boldsymbol x, t) {\rm d} t\\
 &=\int_0^{\theta} \left( \int_0^t \phi (\boldsymbol x, \tau) {\rm d} \tau
\right) \bar p (\boldsymbol x, t) {\rm d} t.
\end{align*}
Next, we take the test function $q =\psi_1$ in \eqref{v1} and get
\begin{align}\label{tf}
 \int_{\Omega_1} \frac{1}{c^2} \partial_t^2 p\,  \bar \psi_1 {\rm d}
\boldsymbol x= -\int_{\Omega_1} \nabla p \cdot \nabla \bar {\psi_1} {\rm d}
\boldsymbol x +\int_{\Gamma_h} (\mathscr T p)\bar \psi_1 {\rm d} \boldsymbol
r\notag\\
+\int_{\Gamma_f} \rho_1(\boldsymbol n \cdot \partial_t^2 \boldsymbol
u) \bar{\psi}_1{\rm d} \gamma.
\end{align}
It follows from \eqref{F1}  and the initial conditions \eqref{trp3}
that
\begin{align*}
 {\rm Re}\int_0^{\theta} \int_{\Omega_1}
 \frac{1}{c^2} \partial_t^2  p \,\bar \psi_1 {\rm d} \boldsymbol  x  {\rm d} t
 &= {\rm Re} \int_{\Omega_1} \int_0^{\theta} \frac{1}{c^2} \left(\partial_t
\left( \partial_t p \,\bar \psi_1\right) +\partial_t p\, \bar p \right){\rm d}
t{\rm d} \boldsymbol x \\
 &={\rm Re} \int_{\Omega_1} \frac{1}{c^2 } \left(  \partial_t p \,\bar {\psi}_1
\big|_{0}^{\theta} +\frac{1}{2} |p |^2 \big|_{0} ^{\theta}\right) {\rm d}
\boldsymbol  x\\
 &= \frac{1}{ 2} \| \frac{1}{c } p (\cdot, \theta)\|^2_{L^2 (\Omega_1)}.
\end{align*}
It is easy to verify that
\begin{align*}
 {\rm Re} \int_0^{\theta} &\int_{\Gamma_f} \rho_1 ( \boldsymbol n \cdot
\partial_t^2 \boldsymbol u) \bar \psi_1 {\rm d} \gamma {\rm d}t
 ={\rm Re } \int_{\Gamma_f} \int_0^{\theta} \rho_1 \left( \partial_t (
\boldsymbol n \cdot \partial_t \boldsymbol u\, \bar \psi_1) +
 (\boldsymbol n \cdot \partial_t \boldsymbol u) \bar p \right) {\rm d}t {\rm
d}\gamma \\
 &={\rm Re}\int_{\Gamma_f} \rho_1 \left( \boldsymbol n \cdot \partial_t
\boldsymbol u\,\bar {\psi}_1 \big|_0^{\theta} \right){\rm d}\gamma
 +{\rm Re}\int_0^{\theta} \int_{\Gamma_f} \rho_1 (\boldsymbol n \cdot \partial_t
\boldsymbol u) \bar p {\rm d} \gamma {\rm d} t\\
 &={\rm Re}\int_0^{\theta} \int_{\Gamma_f} \rho_1 (\boldsymbol n \cdot
\partial_t \boldsymbol u) \bar p {\rm d} \gamma {\rm d} t.
\end{align*}
Integrating \eqref{tf} from $t=0$ to $t=\theta$ and taking the real parts yield
\begin{align}
\frac{1}{2}&\|\frac{1}{c} p (\cdot, \theta)\|^2_{L^2 (\Omega_1)} +{\rm Re} \int_0^{\theta}
\int_{\Omega_1} \nabla p \cdot \bar {\psi}_1 {\rm d} \boldsymbol  x  {\rm d} t
\nonumber \\
 &= \frac{1}{2} \|\frac{1}{c} p (\cdot, \theta)\|^2_{L^2 (\Omega_1)}
+\frac{1}{2} \int_{\Omega_1} \left| \int_0^{\theta} \nabla p (\cdot, t) {\rm d}
t \right|^2 {\rm d} \boldsymbol x \nonumber\\
 &={\rm Re}\int_0^{\theta} \langle \mathscr T p, \psi_1 \rangle_{\Gamma_h} {\rm
d }t +{\rm Re} \int_0^{\theta} \int_{\Gamma_f} \rho_1 (\boldsymbol n \cdot
\partial_t^2 \boldsymbol u)\bar {\psi}_1 {\rm d} \gamma {\rm d} t \nonumber\\
 &= {\rm Re}\int_0^{\theta} \langle \mathscr T p, \psi_1 \rangle_{\Gamma_h} {\rm
d }t +{\rm Re}\int_0^{\theta} \int_{\Gamma_f} \rho_1 (\boldsymbol n \cdot
\partial_t \boldsymbol u) \bar p {\rm d} \gamma {\rm d} t.\label {en1}
\end{align}

We define another auxiliary function
\[
 \boldsymbol \psi_2  (\boldsymbol x, t) = \int_t^{\theta}
\partial_{\tau}\boldsymbol u (\boldsymbol x, \tau) {\rm d} \tau, \quad
\boldsymbol x \in \Omega_2,~~~0 \leq t \leq \theta < T.
\]
Clearly, we have
\begin{align}\label{F3}
 \boldsymbol \psi_2 (\boldsymbol x, \theta) =0, \quad \partial_t  \boldsymbol
\psi_2 (\boldsymbol x, t) = -\partial_t\boldsymbol u (\boldsymbol x, t).
\end{align}
Using the similar proof as that for \eqref{F2},  for any $\boldsymbol \phi
(\boldsymbol x, t) \in L^2 \big(0, \xi;~ L^2 (\Omega_2)^2 \big)$, we may
show that
\begin{align}\label{F4}
 \int_0^{\theta} \boldsymbol \phi (\boldsymbol x, t) \cdot \bar {\boldsymbol
\psi}_2 (\boldsymbol x, t) {\rm d} t=
 \int_0^{\theta} \left( \int_0^t \boldsymbol \phi (\boldsymbol x, \tau) {\rm d}
\tau\right) \cdot  \partial_t\bar { \boldsymbol u} (\boldsymbol x, t) {\rm d} t.
\end{align}
Taking the test function $\boldsymbol v = \boldsymbol \psi_2$ in \eqref{v2}, we
can get
\begin{align}\label{tf2}
 \int_{\Omega_2} \rho_2  \partial_t^2 (\partial_t \boldsymbol u) \cdot \bar
{\boldsymbol \psi}_2 {\rm d} \boldsymbol  x =
 -\int_{\Omega_2} \big( \mu \nabla (\partial_t \boldsymbol u) : \nabla {\bar
{\boldsymbol \psi}}_2 + (\lambda + \mu) (\nabla \cdot (\partial_t\boldsymbol u))
(\nabla \cdot \bar {\boldsymbol \psi}_2)\notag\\
+ \partial_t  \boldsymbol j \cdot \bar {\boldsymbol \psi}_2 \big) {\rm d}
\boldsymbol x-\int_{\Gamma_f} (\partial_t p) (\boldsymbol n \cdot \bar {
\boldsymbol \psi}_2) {\rm d} \gamma.
\end{align}
It follows from \eqref{F3}  and the initial condition in \eqref {TRP} that
\begin{align*}
 {\rm Re} \int_0^{\theta} \int_{\Omega_2}
 \rho_2 \partial_t^2  ( \partial_t \boldsymbol u) \cdot \bar {\boldsymbol \psi
}_2 {\rm d} \boldsymbol  x  {\rm d} t
 =&{\rm Re} \int_{\Omega_2} \int_0^{\theta} \rho_2 \left( \partial_t
(\partial^2_t \boldsymbol u  \cdot \bar {\boldsymbol \psi}_2)
 +\partial^2_t \boldsymbol u \cdot \partial_t\bar {\boldsymbol u} \right) {\rm
d} t{\rm d} \boldsymbol x \\
 =& {\rm Re} \int_{\Omega_2} \rho_2\left(  (\partial^2_t \boldsymbol u \cdot
\bar {\boldsymbol \psi}_2 ) \big|_{0}^{\theta}+
 \frac{1}{2 } |\partial_t \boldsymbol u|^2  \big|_0^{\theta}\right) {\rm d}
\boldsymbol x \\
 =& \frac{\rho_2}{2} \| \partial_t \boldsymbol u (\cdot,~ \theta)\|^2_{L^2
(\Omega_2)^3},
\end{align*}
and
\begin{align*}
 {\rm Re}\int_0^{\theta} \int_{\Gamma_f} (\partial_t p )( \boldsymbol n \cdot
\bar {\boldsymbol \psi}_2) {\rm d} \gamma {\rm d} t
 &={\rm Re} \int_{\Gamma_f} \int_0^{\theta} \left( \partial_t \left( p
\boldsymbol n \cdot \bar {\boldsymbol \psi}_2 \right)
 + p  (\boldsymbol n \cdot \partial_t \bar {\boldsymbol u})  \right){\rm d}
t{\rm d} \gamma\\
 &={\rm Re}\int_{\Gamma_f} \left( p \boldsymbol n \cdot \bar {\boldsymbol
\psi}_2 \right) \big|_0^{\theta} {\rm d} \gamma +
 {\rm Re}\int_0^{\theta} \int_{\Gamma_f} p (\boldsymbol n \cdot \partial_t \bar
{\boldsymbol u}) {\rm d} \gamma {\rm d} t\\
 &={\rm Re}\int_0^{\theta} \int_{\Gamma_f} p (\boldsymbol n \cdot \partial_t
\bar {\boldsymbol u}) {\rm d} \gamma {\rm d} t.
\end{align*}
Integrating \eqref{tf2} from $t=0$ to $t=\theta$ and taking the real parts yield
\begin{align}
 \frac{\rho_2}{2} &\| \partial_t \boldsymbol u (\cdot, \theta)\|^2_{L^2
(\Omega_2)^3} +{\rm Re}\int_0^{\theta}\int_{\Omega_2}  \big( \mu \nabla
(\partial_t \boldsymbol u (\cdot, t)) : \nabla \bar {\boldsymbol \psi}_2 (\cdot,
t)\notag\\
& +(\lambda+\mu) \left(\nabla \cdot (\partial_t \boldsymbol u (\cdot, t)))
(\nabla \cdot \bar {\boldsymbol \psi}_2 (\cdot, t) \right)\big) {\rm d}
\boldsymbol  x{\rm d }t \nonumber\\
 =&\frac{\rho_2}{2}\| \partial_t\boldsymbol u (\cdot, \theta)\|^2_{L^2
(\Omega_2)^2}
 +\frac{1}{2} \int_{\Omega_2} \bigg( \mu \left| \int_0^{\theta} \nabla
(\partial_t \boldsymbol u (\cdot, t)) {\rm d} t \right|^2_{F}\notag\\
&+ (\lambda+\mu) \left| \int_0^{\theta} \nabla \cdot (\partial_t  \boldsymbol u
(\cdot, t)) {\rm d} t \right|^2\bigg) {\rm d} \boldsymbol  x \nonumber\\
 =& - {\rm Re}\int_0^{\theta} \int_{\Omega_2} \partial_t \boldsymbol j \cdot
\bar {\boldsymbol \psi}_2 {\rm d} \boldsymbol x  {\rm d}t
 -{\rm Re}\int_0^{\theta} \int_{\Gamma_f} p (\boldsymbol n \cdot \partial_t \bar
{\boldsymbol u}) {\rm d} \gamma {\rm d} t,  \label{en2}
\end{align}
where
\[
 \left| \int_0^{\theta} \nabla ( \partial_t \boldsymbol u (\cdot, t)) {\rm d}
t\right|^2_{F}: = \int_0^{\theta} \nabla (\partial_t \boldsymbol u (\cdot, t))
{\rm d }t :
\int_0^{\theta} \nabla (\partial_t \bar {\boldsymbol u} (\cdot, t)) {\rm d} t.
\]
Multiplying \eqref{en2} by $\rho_1$ and adding it to \eqref{en1} give
\begin{align}
 \frac{1}{2}
 &\|\frac{1}{c} p (\cdot, \theta)\|^2_{L^2 (\Omega_1)} +\frac{1}{2}
\int_{\Omega_1} \left| \int_0^{\theta} \nabla p (\cdot, t) {\rm d} t \right|^2
{\rm d} \boldsymbol  x
 + \frac{\rho_1 \rho_2}{2}\| \partial_t\boldsymbol u (\cdot, \theta)\|^2_{L^2
(\Omega_2)^3}\nonumber\\
 &+\frac{\rho_1}{2} \int_{\Omega_2} \left( \mu \left| \int_0^{\theta} \nabla
(\partial_t \boldsymbol u (\cdot, t)) {\rm d} t \right|^2_{F} +
 (\lambda+\mu) \left| \int_0^{\theta} \nabla \cdot (\partial_t  \boldsymbol u
(\cdot, t)) {\rm d} t \right|^2\right) {\rm d} \boldsymbol  x \nonumber\\
 =&{\rm Re}\int_0^{\theta} \langle \mathscr T p, \psi_1 \rangle_{\Gamma_h} {\rm
d }t +{\rm Re}\int_0^{\theta} \int_{\Gamma_f} \rho_1 (\boldsymbol n \cdot
\partial_t \boldsymbol u)\bar p {\rm d} \gamma {\rm d} t\nonumber\\
 &- {\rm Re}\int_0^{\theta} \int_{\Omega_2} \rho_1(\partial_t \boldsymbol j
\cdot \bar {\boldsymbol \psi}_2) {\rm d} \boldsymbol x  {\rm d}t
 -{\rm Re}\int_0^{\theta} \int_{\Gamma_f} \rho_1 p (\boldsymbol n \cdot
\partial_t \bar {\boldsymbol u}) {\rm d} \gamma {\rm d} t \nonumber\\
 =&{\rm Re}\int_0^{\theta} \langle \mathscr T p, \psi_1 \rangle_{\Gamma_h} {\rm
d }t - {\rm Re}\int_0^{\theta} \int_{\Omega_2} \rho_1(\partial_t \boldsymbol j
\cdot\bar {\boldsymbol \psi}_2) {\rm d} \boldsymbol x  {\rm d}t.\label{int}
\end{align}

In what follows, we estimate the two terms on the right-hand side of \eqref{int}
separately. Using Lemma \ref{ttp} and \eqref{F2}, we obtain
\begin{align}
 {\rm Re} \int_0^{\theta} \langle \mathscr T p, \psi_1 \rangle_{\Gamma_h} {\rm d
}t &= {\rm Re}\int_0^{\theta}\int_{\Gamma_h} (\mathscr T p) \bar {\psi}_1 {\rm
d}\boldsymbol r {\rm d} t \nonumber\\
 &={\rm Re} \int_{\Gamma_h}\int_0^{\theta} \left( \int_0^t \mathscr T p (\cdot,
\tau) {\rm d} \tau \right) \bar p (\cdot, t) {\rm d} t{\rm d}\boldsymbol r
\leq 0.\label{p1}
\end{align}
For $0 \leq t \leq \theta \leq T$, we have from \eqref{F4} that
\begin{align}
 {\rm Re}\int_0^{\theta} \int_{\Omega_2}& -\rho_1(\partial_t \boldsymbol j \cdot
\bar {\boldsymbol \psi}_2) {\rm d} \boldsymbol x  {\rm d}t
 = {\rho_1}{\rm Re}\int_{\Omega_2} \int_0^{\theta} \left( \int_0^t- \partial_t
\boldsymbol j (\cdot, \tau) {\rm d} \tau \right) \cdot
 \partial_t \bar {\boldsymbol u}  (\cdot, t) {\rm d}t {\rm d} \boldsymbol x
\nonumber\\
 &= \rho_1 {\rm Re}\int_0^{\theta} \int_0^t \int_{\Omega_2} - \partial_t
\boldsymbol j (\cdot, \tau) \cdot \partial_t \bar {\boldsymbol u} (\cdot, t)
{\rm d} \boldsymbol x {\rm d}\tau {\rm d} t \nonumber\\
 &\leq \rho_1 \int_0^{\theta} \left( \int_0^{t} \|\partial_t  \boldsymbol j
(\cdot, \tau)\|_{L^2 (\Omega_2)^3} {\rm d }\tau\right)
 \| \partial_t  \boldsymbol u (\cdot, t)\|_{L^2 (\Omega_2)^2} {\rm d} t
\nonumber\\
 &\leq \rho_1 \int_0^{\theta} \left( \int_0^{\theta} \|\partial_t  \boldsymbol j
(\cdot, t)\|_{L^2 (\Omega_2)^3} {\rm d }t\right)
 \| \partial_t  \boldsymbol u (\cdot, t)\|_{L^2 (\Omega_2)^3} {\rm d} t
\nonumber\\
 &\leq \rho_1 \left( \int_0^{\theta} \|\partial_t  \boldsymbol j (\cdot,
t)\|_{L^2(\Omega_2)^3} {\rm d }t\right)
 \left( \int_0^{\theta} \|\partial_t  \boldsymbol u (\cdot, t)\|_{L^2
(\Omega_2)^3} {\rm d }t\right). \label{p2}
\end{align}
Substituting \eqref{p1} and \eqref{p2} into \eqref{int}, we have for any $\theta\in [0,T]$ that
\begin{align}
 \frac{1}{2}&\|\frac{1}{c} p (\cdot, \theta)\|^2_{L^2 (\Omega_1)}
 + \frac{\rho_1 \rho_2}{2}\| \partial_t\boldsymbol u (\cdot, \theta)\|^2_{L^2
(\Omega_2)^3} \nonumber\\
 \leq &\frac{1}{2}
 \|\frac{1}{c} p (\cdot, \theta)\|^2_{L^2 (\Omega_1)} +\frac{1}{2}
\int_{\Omega_1} \left| \int_0^{\theta} \nabla p (\cdot, t) {\rm d} t \right|^2
{\rm d} \boldsymbol  x
 + \frac{\rho_1 \rho_2}{2}\| \partial_t\boldsymbol u (\cdot, \theta)\|^2_{L^2
(\Omega_2)^3}\nonumber\\
 &+\frac{\rho_1}{2} \int_{\Omega_2} \left( \mu \left| \int_0^{\theta} \nabla
(\partial_t \boldsymbol u (\cdot, t)) {\rm d} t \right|^2_{F} +
 (\lambda+\mu) \left| \int_0^{\theta} \nabla \cdot (\partial_t  \boldsymbol u
(\cdot, t)) {\rm d} t \right|^2\right) {\rm d} \boldsymbol  x \nonumber\\
  \leq & \rho_1\left( \int_0^{\theta} \|\partial_t  \boldsymbol j (\cdot,
t)\|_{L^2 (\Omega_2)^3} {\rm d }t\right)
 \left( \int_0^{\theta} \|\partial_t  \boldsymbol u (\cdot, t)\|_{L^2
(\Omega_2)^2} {\rm d }t\right).\label{p3}
\end{align}

Taking the $L^{\infty}$ norm with respect to $\theta$ on both sides of
\eqref{p3} yields
\begin{align*}
 \|p\|^2_{L^{\infty}\left(0, T;~ L^2(\Omega_1) \right)}& +\|\partial_t
\boldsymbol u\|^2_{L^{\infty}\left(0, T;~ L^2(\Omega_2)^3 \right)}\\
 &\lesssim T \left( \|\partial_t \boldsymbol j\|_{L^1(0, T; ~L^2 (\Omega_2)^3)}
\|\partial_t \boldsymbol u\|_{L^{\infty}\left(0, T;~ L^2(\Omega_2)^3
\right)}\right)
\end{align*}
Applying the Young inequality to the above inequality, we get
\begin{align}\label{yi}
  \|p\|^2_{L^{\infty}(0, T;~ L^2(\Omega_1))} +\|\partial_t \boldsymbol
u\|^2_{L^{\infty}(0, T;~ L^2(\Omega_2)^3 )}\lesssim T^2  \|\partial_t
\boldsymbol j\|^2_{L^1(0, T; ~L^2 (\Omega_2)^3)}.
\end{align}
It follows from the Cauchy--Schwarz inequality that
\begin{align*}
 \|p\|_{L^{\infty} (0, T; L^2 (\Omega_1))} &\leq \|p\|_{L^{\infty}\left(0, T;~
L^2(\Omega_1) \right)} +\|\partial_t \boldsymbol u\|_{L^{\infty}\left(0, T;~
L^2(\Omega_2)^3 \right)}\\
 &\lesssim T  \|\partial_t \boldsymbol j\|_{L^1(0, T; ~L^2 (\Omega_2)^3)},
\end{align*}
which gives the estimate \eqref{ess1}.

For the elastic displacement $\boldsymbol u$,  using the epsilon inequality
gives
\begin{align*}
 \| \boldsymbol u\|^2_{L^{\infty} (0, T; L^2 (\Omega_2)^3)}& =\int_0^{t}
\partial_{\tau} \| \boldsymbol u (\cdot, \tau)\|^2_{L^{\infty} (0, T; L^2
(\Omega_2)^3)} {\rm d \tau}\\
& \leq \epsilon T \| \boldsymbol u\|^2_{L^{\infty} (0, T; L^2 (\Omega_2)^3)}+
\frac{T}{\epsilon}
  \|\partial_t \boldsymbol u\|^2_{L^{\infty}\left(0, T;~ L^2(\Omega_2)^3
\right)}.
\end{align*}
Choosing $\epsilon =\frac{1}{2T}$, we have from  \eqref{yi} that
\begin{align*}
 \| \boldsymbol u\|^2_{L^{\infty} (0, T; L^2 (\Omega_2)^3)}
  &\lesssim T^2  \|\partial_t \boldsymbol u\|^2_{L^{\infty}\left(0, T;~
L^2(\Omega_2)^3 \right)}\\
 &\lesssim T^2 \left( \|p\|^2_{L^{\infty}\left(0, T;~ L^2(\Omega_1) \right)}
+\|\partial_t \boldsymbol u\|^2_{L^{\infty}\left(0, T;~ L^2(\Omega_2)^3 \right)}
\right)\\
 & \lesssim T^4  \|\partial_t \boldsymbol j\|^2_{L^1(0, T; ~L^2 (\Omega_2)^3)},
\end{align*}
which implies the estimate \eqref{ess2}.

Integrating \eqref{p3} with respect to $\theta$ from $0$ to $T$ and using the
Cauchy--Schwarz inequality, we obtain
\begin{align*}
 \|p\|^2_{L^{2}\left(0, T;~ L^2(\Omega_1) \right)}& +\|\partial_t \boldsymbol
u\|^2_{L^{2}\left(0, T;~ L^2(\Omega_2)^3 \right)}\\
 &\lesssim T^{3/2} \left( \|\partial_t \boldsymbol j\|_{L^1(0, T; ~L^2
(\Omega_2)^3)} \|\partial_t \boldsymbol u\|_{L^{2}\left(0, T;~ L^2(\Omega_2)^3
\right)}\right).
\end{align*}
Using Young's inequality again to the above equation yields
 \begin{align}\label{yi2}
  \|p\|^2_{L^{2}\left(0, T;~ L^2(\Omega_1) \right)} +\|\partial_t \boldsymbol
u\|^2_{L^{2}\left(0, T;~ L^2(\Omega_2)^3 \right)}
  \lesssim T^3  \|\partial_t \boldsymbol j\|^2_{L^1(0, T; ~L^2 (\Omega_2)^3)}.
 \end{align}
It follows from the Cauchy--Schwarz inequality that
\begin{align*}
 \|p\|_{L^{2}\left(0, T;~ L^2(\Omega_1) \right)}
 &\leq  \|p\|_{L^{2}\left(0, T;~ L^2(\Omega_1) \right)} +\|\partial_t
\boldsymbol u\|_{L^{2}\left(0, T;~ L^2(\Omega_2)^3 \right)} \\
&\lesssim T^{3/2} \|\partial_t \boldsymbol j\|_{L^1(0, T; ~L^2 (\Omega_2)^3)},
\end{align*}
which shows the estimate \eqref{ess}.

Taking $\epsilon =\frac{1}{2T}$ and applying the epsilon inequality, we have
\begin{align*}
 \| \boldsymbol u\|^2_{L^2 (0, T; L^2 (\Omega_2)^3)}& =\int_0^{t}
\partial_{\tau} \| \boldsymbol u (\cdot, \tau)\|^2_{L^2(0, T; L^2
(\Omega_2)^3)} {\rm d \tau}\\
& \leq \frac{1}{2} \| \boldsymbol u\|^2_{L^2 (0, T; L^2 (\Omega_2)^3)}+ 2T^2
  \|\partial_t \boldsymbol u\|^2_{L^2\left(0, T;~ L^2(\Omega_2)^3
\right)},
\end{align*}
It follows from the above inequality and \eqref{yi2} that
\begin{align*}
 \| \boldsymbol u\|^2_{L^2 (0, T; L^2 (\Omega_2)^3)}
  \lesssim T^2  \|\partial_t \boldsymbol u\|^2_{L^2\left(0, T;~
L^2(\Omega_2)^3 \right)}\lesssim T^5
\|\partial_t \boldsymbol j\|^2_{L^1(0, T; ~L^2 (\Omega_2)^3)},
\end{align*}
which implies the estimate \eqref{ess3}.
\end{proof}

\section{Conclusion}\label{cl}

In this paper we have studied the time-domain acoustic-elastic interaction
problem in an unbounded structure in the three-dimensional space. The problem
models the wave propagation in a two-layered medium consisting of the air/fluid
and the solid due to an active source located in the solid. We reduce the
scattering problem into an initial-boundary value problem by using the exact
TBC. We establish the well-posedness and the stability for the variational
problem in the $s$-domain. In the time-domain, we show that the reduced
problem has a unique weak solution by using the energy method. The main
ingredients of the proofs are the Laplace transform, the Lax--Milgram lemma,
and the Parseval identity. We also obtain a priori estimates with explicit time
dependence for the quantities of acoustic wave pressure and elastic wave
displacement by taking special test functions to the time-domain variational
problem.

\end{document}